\documentclass[a4paper,12pt]{article}
\usepackage{latexsym}
\usepackage{amsmath}
\usepackage{amssymb}
\usepackage{enumerate}
\usepackage{bm}
\usepackage{mathrsfs}

\usepackage{theorem}
\newtheorem{theorem}{Theorem}[section]
\newtheorem{proposition}[theorem]{Proposition}
\newtheorem{lemma}[theorem]{Lemma}
\newtheorem{corollary}[theorem]{Corollary}

\theorembodyfont{\rmfamily}
\newtheorem{proof}{\textmd{\textit{Proof.}}}

\newtheorem{outline}{\textmd{\textit{Outline of proof.}}}

\newtheorem{remark}[theorem]{Remark}

\newtheorem{definition}[theorem]{Definition}

\makeatletter

\@addtoreset{equation}{section}
\makeatother

\newcommand{\qedd}{\hfill \Box}
\newcommand{\ve}{\varepsilon}
\newcommand{\del}{\partial}
\newcommand{\lra}{\longrightarrow}
\newcommand{\e}{\mathrm{e}}

\newcommand{\Z}{\ensuremath{\mathbb{Z}}}
\newcommand{\R}{\ensuremath{\mathbb{R}}}
\newcommand{\Sph}{\ensuremath{\mathbb{S}}}
\newcommand{\cC}{\ensuremath{\mathcal{C}}}
\newcommand{\cE}{\ensuremath{\mathcal{E}}}
\newcommand{\cI}{\ensuremath{\mathcal{I}}}
\newcommand{\cL}{\ensuremath{\mathcal{L}}}
\newcommand{\cP}{\ensuremath{\mathcal{P}}}
\newcommand{\scN}{\ensuremath{\mathscr{N}}}
\newcommand{\fm}{\ensuremath{\mathfrak{m}}}
\newcommand{\bs}{\ensuremath{\mathbf{s}}}
\newcommand{\bK}{\ensuremath{\mathbf{K}}}
\newcommand{\bS}{\ensuremath{\mathbf{S}}}
\newcommand{\sC}{\ensuremath{\mathsf{C}}}
\newcommand{\sS}{\ensuremath{\mathsf{S}}}

\def\vol{\mathop{\mathrm{vol}}\nolimits}
\def\diam{\mathop{\mathrm{diam}}\nolimits}
\def\div{\mathop{\mathrm{div}}\nolimits}

\def\loc{\mathop{\mathrm{loc}}\nolimits}
\def\Cut{\mathop{\mathrm{Cut}}\nolimits}
\def\Ent{\mathop{\mathrm{Ent}}\nolimits}

\def\Var{\mathop{\mathrm{Var}}\nolimits}
\def\Ric{\mathop{\mathrm{Ric}}\nolimits}

\def\CD{\mathop{\mathrm{CD}}\nolimits}
\def\RCD{\mathop{\mathrm{RCD}}\nolimits}
\def\HS{\mathop{\mathrm{HS}}\nolimits}

\newcommand{\Grad}{\bm{\nabla}}
\newcommand{\Lap}{\bm{\Delta}}

\newcommand{\rev}[1]{\overleftarrow{#1}}

\title{Nonlinear geometric analysis on Finsler manifolds}
\author{Shin-ichi Ohta\thanks{Department of Mathematics, Kyoto University,
Kyoto, 606-8502, Japan ({\sf sohta@math.kyoto-u.ac.jp}),
{\it Current address}: Department of Mathematics, Osaka University,
Osaka, 560-0043, Japan ({\sf s.ohta@math.sci.osaka-u.ac.jp});
Supported in part by JSPS Grant-in-Aid for Scientific Research (KAKENHI) 15K04844.}}
\date{}
\begin{document}

\maketitle

\begin{abstract}
This is a survey article on recent progress of comparison geometry
and geometric analysis on Finsler manifolds of weighted Ricci curvature bounded below.
Our purpose is two-fold:
Give a concise and geometric review on the birth of weighted Ricci curvature
and its applications;
Explain recent results from a nonlinear analogue of the $\Gamma$-calculus
based on the Bochner inequality.
In the latter we discuss some gradient estimates, functional inequalities,
and isoperimetric inequalities.
%
\end{abstract}

\tableofcontents

\section{Introduction}

The aim of this article is to review the recent developments
of comparison geometry and geometric analysis
on Finsler manifolds of weighted Ricci curvature bounded below.
The weighted Ricci curvature was introduced by the author in \cite{Oint}.
Since then it has helped us to understand the similarities and differences
between Riemannian and Finsler manifolds more deeply.
One of the main challenges, compared with the Riemannian case,
is in the nonlinearity of the Laplacian and the associated heat equation.
How to deal with such a nonlinear evolution equation
would be of interest also in the analytic viewpoint.

A \emph{Finsler manifold} will be a pair of a manifold $M$
and a nonnegative function $F$ on the tangent bundle $TM$
such that $F|_{T_xM}$ gives a Minkowski norm for each $x \in M$.
We remark that $F(-v) \neq F(v)$ is allowed,
which is called the \emph{non-reversibility} and
is a special feature of Finsler manifolds.
One can define the distance and geodesics in natural geometric ways,
whereas the non-reversibility of $F$ leads to the \emph{asymmetry} of the distance,
namely $d(y,x)$ does not necessarily coincide with $d(x,y)$.
Analyzing the behavior of geodesics, we can further introduce
Jacobi fields and the curvature tensor.
Thus we arrive at the natural notions of the \emph{flag curvature}
(a generalization of the Riemannian sectional curvature)
and the \emph{Ricci curvature}.
Some results in comparison Riemannian geometry concerning the distance
are then generalized to the Finsler context
(for example, the Cartan--Hadamard theorem and the Bonnet--Myers theorem).

Now, to proceed further in this direction,
what is a \emph{measure} controlled by the Ricci curvature?
At this point we have another difficulty;
There are several choices for a `canonical' measure
(such as the Busemann--Hausdorff measure and the Holmes--Thompson measure),
all of them go down to the volume measure on a Riemannian manifold.
Our strategy initiated in \cite{Oint} is that,
instead of choosing some constructive measure,
we begin with an arbitrary (smooth, positive) measure $\fm$ on $M$
and modify the Ricci curvature according to the choice of $\fm$.
The modification is done in the same manner as weighted Riemannian manifolds.
We call this modified Ricci curvature the \emph{weighted Ricci curvature} $\Ric_N$
(also called the \emph{Bakry--\'Emery--Ricci curvature}),
where $N$ is a parameter sometimes called the \emph{effective dimension}.
By means of $\Ric_N$, one can generalize many results including
the \emph{Bishop--Gromov volume comparison} to this Finsler context and,
most notably, the \emph{curvature-dimension condition} $\CD(K,N)$
in the sense of Lott--Sturm--Villani (\cite{StI,StII,LV}) is equivalent to
the lower curvarture bound $\Ric_N \ge K$ (\cite{Oint}).
This deep equivalence relation ensures
the naturalness and importance of the weighted Ricci curvature.

The Ricci curvature plays the prominent role in geometric analysis.
The \emph{Laplacian} $\Lap$ acting on functions is naturally defined as
the divergence (in terms of $\fm$) of the gradient vector field (in terms of $F$).
This is also called the \emph{Witten Laplacian} or the \emph{weighted Laplacian}.
Since taking the gradient vector (more precisely, the Legendre transform)
is a nonlinear operation, our Laplacian is \emph{nonlinear}
(unless $F$ comes from a Riemannian metric).
Nonetheless, since the Laplacian is (locally) uniformly elliptic,
the associated \emph{heat equation} $\del_t u=\Lap u$ is still well-posed.
In fact, we can apply the classical technique due to Saloff-Coste \cite{Sal} and others
to show the unique existence and a certain regularity of solutions to the heat equation
(\cite{GS,OShf}).

Furthermore, the \emph{Bochner--Weitzenb\"ock formula} was established in \cite{OSbw}
in the form of
\[ \Delta\!^{\Grad u} \bigg[ \frac{F^2(\Grad u)}{2} \bigg] -D[\Lap u](\Grad u)
 =\Ric_{\infty}(\Grad u) +\| \Grad^2 u \|_{\HS(\Grad u)}^2, \]
followed by the \emph{Bochner inequality}
\[ \Delta\!^{\Grad u} \bigg[ \frac{F^2(\Grad u)}{2} \bigg] -D[\Lap u](\Grad u)
 \ge KF^2(\Grad u) +\frac{(\Lap u)^2}{N} \]
under the lower curvature bound $\Ric_N \ge K$.
Here we mixed the nonlinear Laplacian $\Lap$ and its linearization $\Delta\!^{\Grad u}$.
Similarly to the Riemannian case, the Bochner inequality plays a significant role.
See \cite{OSbw,Osplit,WX,Xi,YH}
for some geometric and analytic applications not covered in this article.

Our focus in this article will be on a nonlinear analogue of the \emph{$\Gamma$-calculus}.
The $\Gamma$-calculus, developed by Bakry and his collaborators
(see \cite{BE,Ba} and the recent book \cite{BGL}),
is a successful theory of analyzing linear diffusion operators and the associated semigroups
by means of the integration by parts and a kind of Bochner inequality.
This theory fits surprisingly well with our Finsler setting,
although the Laplacian is nonlinear.

We will discuss three applications of the $\Gamma$-calculus:
gradient estimates, functional inequalities, and isoperimetric inequalities.
We will give the outlines of some proofs to show the basic idea of the $\Gamma$-calculus.
We first consider the $L^2$- and $L^1$-\emph{gradient estimates}
for heat flow under $\Ric_{\infty} \ge K$.
These gradient estimates are indeed equivalent to $\Ric_{\infty} \ge K$
similarly to the Riemannian situation (\cite{vRS,Oisop}).
Next we show three functional inequalities:
the \emph{Poincar\'e--Lichnerowicz $($spectral gap$)$ inequality},
the \emph{logarithmic Sobolev inequality}, and the \emph{Sobolev inequality}.
Here the assumption is $\Ric_N \ge K>0$ with $N \in [n,\infty]$,
and $N$ can be also negative only in the Poincar\'e--Lichnerowicz inequality (\cite{Ofunc}).
We finally study the \emph{Gaussian isoperimetric inequality}
under $\Ric_{\infty} \ge K>0$, which generalizes Bakry--Ledoux's inequality
in the Riemannian situation (\cite{BL,Oisop}).
All the estimates are sharp and have the same forms as Riemannian manifolds.
Only the exception is the admissible range of the exponent $p$
in the Sobolev inequality (see Remark~\ref{rm:Sobo}).
We stress that, however, we cannot generalize every result of Riemannian manifolds
to Finsler manifolds.
One of the most important differences is the lack of contraction property of heat flow
(\cite{OSnc}).
In general we know much less about gradient flows of convex functions
in Finsler manifolds or even in normed spaces (see \cite{OP} for a related discussion).

The organization of the article is as follows.
Section~\ref{sc:prel} is a crash course for the basic notions and ideas of comparison Finsler geometry.
We try to explain in a geometric way how the distance, geodesics, and curvatures arise.
Section~\ref{sc:heat} is devoted to the study of the Laplacian and the heat semigroup.
Then we consider the gradient estimates in Section~\ref{sc:grades},
the functional inequalities in Section~\ref{sc:funct},
and the isoperimetric inequality in Section~\ref{sc:isop}, respectively.
\medskip

{\it Acknowledgements.}
I thank Professor Sorin Sabau for encouraging me to write a survey on this subject.
I am also grateful to an anonymous referee for valuable comments.

\section{Comparison geometry of Finsler manifolds}\label{sc:prel}

This section is devoted to a concise review on comparison geometry of Finsler manifolds.
The main object will be a Finsler manifold equipped with a measure
whose weighted Ricci curvature is bounded from below.
We refer to \cite{BCS,Shlec,SS} for the fundamentals of Finsler geometry
as well as related comparison geometric studies.

Throughout the article, unless otherwise indicated,
let $M$ be a connected $\cC^{\infty}$-manifold without boundary
of dimension $n \ge 2$.
We will also fix an arbitrary positive $\cC^{\infty}$-measure $\fm$ on $M$
from Subsection~\ref{ssc:wRic}.

\subsection{Finsler structures}\label{ssc:Fmfd}

Given local coordinates $(x^i)_{i=1}^n$ on an open set $U \subset M$,
we will always use the fiber-wise linear coordinates
$(x^i,v^j)_{i,j=1}^n$ of $TU$ such that
\[ v=\sum_{j=1}^n v^j \frac{\del}{\del x^j}\Big|_x \in T_xM, \qquad x \in U. \]

\begin{definition}[Finsler structures]\label{df:Fstr}
We say that a nonnegative function $F:TM \lra [0,\infty)$ is
a \emph{$\cC^{\infty}$-Finsler structure} of $M$ if the following three conditions hold:
\begin{enumerate}[(1)]
\item(\emph{Regularity})
$F$ is $\cC^{\infty}$ on $TM \setminus 0$,
where $0$ stands for the zero section;

\item(\emph{Positive $1$-homogeneity})
It holds $F(cv)=cF(v)$ for all $v \in TM$ and $c \ge 0$;

\item(\emph{Strong convexity})
The $n \times n$ matrix
\begin{equation}\label{eq:gij}
\big( g_{ij}(v) \big)_{i,j=1}^n :=
 \bigg( \frac{1}{2}\frac{\del^2 (F^2)}{\del v^i \del v^j}(v) \bigg)_{i,j=1}^n
\end{equation}
is positive-definite for all $v \in TM \setminus 0$.
\end{enumerate}
We call such a pair $(M,F)$ a \emph{$\cC^{\infty}$-Finsler manifold}.
\end{definition}

Notice that the positive-definiteness is independent of
the choice of local coordinates.
One can similarly define $\cC^l$-Finsler manifolds,
though we consider only $\cC^{\infty}$-Finsler manifolds in this article.
Some more remarks on Definition~\ref{df:Fstr} are ready.

\begin{remark}\label{rm:Fstr}
(a)
The homogeneity (2) is imposed only in the positive direction,
that is, for nonnegative $c$.
This leads to the asymmetry of the associated distance function
(see the next subsection).
If $F(-v)=F(v)$ holds for all $v \in TM$, then we say that $F$ is \emph{reversible}
or \emph{absolutely homogeneous}.

(b)
The strong convexity (3) means that the unit sphere $U_xM:=T_xM \cap F^{-1}(1)$
(called the \emph{indicatrix}) is positively curved, and it implies the strict convexity:
$F(v+w) \le F(v)+F(w)$ for all $v,w \in T_xM$ and equality holds
only when $v=aw$ or $w=av$ for some $a \ge 0$.

(c)
Although we will discuss only under Definition~\ref{df:Fstr},
the $\cC^{\infty}$-regularity (1) and the strong convexity (3) can be weakened
in various ways occasionally (see \cite{OShf} for instance).
\end{remark}

Let us continue the study of the strong convexity in the remainder of this subsection.
Given each $v \in T_xM \setminus 0$, the positive-definite matrix
$(g_{ij}(v))_{i,j=1}^n$ in \eqref{eq:gij} induces
the Riemannian structure $g_v$ of $T_xM$ as
\begin{equation}\label{eq:gv}
g_v\bigg( \sum_{i=1}^n a_i \frac{\del}{\del x^i}\Big|_x,
 \sum_{j=1}^n b_j \frac{\del}{\del x^j}\Big|_x \bigg)
 := \sum_{i,j=1}^n g_{ij}(v) a_i b_j.
\end{equation}
Notice that this definition is coordinate-free, and we have $g_v(v,v)=F^2(v)$.
One can regard $g_v$ as the best Riemannian approximation of $F|_{T_xM}$
in the direction $v$.
In fact, the unit sphere of $F|_{T_xM}$ (which is positively curved due to Remark~\ref{rm:Fstr}(b))
is tangent to that of $g_v$ at $v/F(v)$ up to the second order.
The metric $g_v$ plays quite important roles in comparison Finsler geometry.

In the coordinates $(x^i,\alpha_j)_{i,j=1}^n$ of $T^*U$ given by
$\alpha=\sum_{j=1}^n \alpha_j dx^j$,
we will also consider
\[ g^*_{ij}(\alpha) :=\frac{1}{2} \frac{\del^2[(F^*)^2]}{\del \alpha_i \del \alpha_j}(\alpha),
\qquad i,j=1,2,\ldots,n, \]
for $\alpha \in T^*U \setminus 0$.
Here $F^*:T^*M \lra [0,\infty)$ is the \emph{dual Minkowski norm} to $F$, namely
\[ F^*(\alpha) :=\sup_{v \in T_xM,\, F(v) \le 1} \alpha(v)
 =\sup_{v \in T_xM,\, F(v)=1} \alpha(v) \]
for $\alpha \in T_x^*M$.
It is clear by definition that $\alpha(v) \le F^*(\alpha) F(v)$, and hence
\[ \alpha(v) \ge -F^*(\alpha)F(-v), \qquad \alpha(v) \ge -F^*(-\alpha)F(v). \]
We remark and stress that, however, $\alpha(v) \ge -F^*(\alpha)F(v)$ does not hold in general
due to the non-reversibility of $F$.

The strong convexity is related to the following quantities,
these are fundamental in the geometry of Banach spaces (see \cite{BCL,Ouni}).
For $x \in M$, we define the ($2$-)\emph{uniform smoothness constant} at $x$ by
\[ \sS_F(x) :=\sup_{v,w \in T_xM \setminus 0} \frac{g_v(w,w)}{F^2(w)}
 =\sup_{\alpha,\beta \in T^*_xM \setminus 0}
 \frac{F^*(\beta)^2}{g^*_{\alpha}(\beta,\beta)}. \]
Since $g_v(w,w) \le \sS_F(x) F^2(w)$ and $g_v$ is the Hessian of $F^2/2$ at $v$,
the constant $\sS_F(x)$ measures the concavity of $F^2$ in $T_xM$.
We also set $\sS_F:=\sup_{x \in M} \sS_F(x)$.
Notice that $\sS_F \in [1,\infty]$ and $\sS_F=1$ holds if and only if
$F$ comes from a Riemannian metric (see \cite{Ouni}).
We similarly define the ($2$-)\emph{uniform convexity constants} as
\[ \sC_F(x) :=\sup_{v,w \in T_xM \setminus 0} \frac{F^2(w)}{g_v(w,w)}
 =\sup_{\alpha,\beta \in T^*_xM \setminus 0}
 \frac{g^*_{\alpha}(\beta,\beta)}{F^*(\beta)^2}, \qquad
 \sC_F:=\sup_{x \in M} \sC_F(x). \]
Again, $\sC_F \in [1,\infty]$ in general
and $\sC_F=1$ holds if and only if $(M,F)$ is Riemannian.

It is readily seen that the constants $\sS_F$ and $\sC_F$
control the \emph{reversibility constant}, defined by
\[ \Lambda_F :=\sup_{v \in TM \setminus 0} \frac{F(v)}{F(-v)} \,\in [1,\infty], \]
as follows (see \cite[Lemma~2.4]{Oisop} for instance).

\begin{lemma}\label{lm:rev}
We have
\[ \Lambda_F \le \min\{ \sqrt{\sS_F},\sqrt{\sC_F} \}. \]
\end{lemma}

For later convenience, we introduce the following notations.

\begin{definition}[Reverse Finsler structures]\label{df:rev}
We define the \emph{reverse Finsler structure} $\rev{F}$ of $F$ by
$\rev{F}(v):=F(-v)$.
We will put an arrow $\leftarrow$ on those quantities associated with $\rev{F}$.
\end{definition}

For example, $\rev{d}\!(x,y)=d(y,x)$ (see \S \ref{ssc:geod}),
$\rev{\Ric}(v)=\Ric(-v)$ (\S \ref{ssc:curv}),
$\rev{\Ric}_N(v)=\Ric_N(-v)$ (\S \ref{ssc:wRic}),
$\rev{\Grad}u=-\Grad(-u)$ and $\rev{\Lap} u=-\Lap(-u)$ (\S \ref{ssc:Lap}).

\subsection{Asymmetric distance and geodesics}\label{ssc:geod}

For $x,y \in M$, we define the (\emph{asymmetric}) \emph{distance} from $x$ to $y$ by
\[ d(x,y):=\inf_{\eta} \int_0^1 F\big( \dot{\eta}(t) \big) \,dt, \]
where the infimum is taken over all piecewise $\cC^1$-curves $\eta:[0,1] \lra M$
such that $\eta(0)=x$ and $\eta(1)=y$.
Note that the asymmetry $d(y,x) \neq d(x,y)$ can occur
since $F$ is only positively homogeneous
(so $d$ is not properly a distance function in the usual sense).
We also remark that the squared distance function $d^2(x,\cdot)$
is only $\cC^1$ at $x$ in general,
and that $d^2(x,\cdot)$ is $\cC^2$ at $x$ for all $x \in M$
if and only if $F$ comes from a Riemannian metric (see \cite[Proposition~2.2]{Shvol}).
This is a reason why we have the less regularity for heat flow
(see Theorem~\ref{th:hf} below).

In the manner of metric geometry,
we say that a $\cC^{\infty}$-curve $\eta:I \lra M$ from an interval $I \subset \R$
is \emph{geodesic} if it is locally minimizing and has a constant speed with respect to $d$.
Precisely, there is $C \ge 0$ and, for any $t \in I$, we find $\ve>0$ such that
\[ d\big( \eta(s),\eta(s') \big) =C(s'-s) \]
for all $s,s' \in I \cap [t-\ve,t+\ve]$ with $s \le s'$
(then $F(\dot{\eta}) \equiv C$ as a matter of course).

One can write down the \emph{geodesic equation}
as the Euler--Lagrange equation for the action induced from $F^2$.
In such calculations, the following basic theorem plays fundamental roles
(see \cite[Theorem~1.2.1]{BCS}).

\begin{theorem}[Euler's homogeneous function theorem]\label{th:Euler}
Consider a differentiable function $H:\R^n \setminus \{0\} \lra \R$ satisfying
$H(cv)=c^r H(v)$ for some $r \in \R$ and all $c>0$ and $v \in \R^n \setminus \{0\}$
$($that is, $H$ is \emph{positively $r$-homogeneous}$)$.
Then we have
\[ \sum_{i=1}^n \frac{\del H}{\del v^i}(v)v^i=rH(v) \]
for all $v \in \R^n \setminus \{0\}$.
\end{theorem}

Observe that $g_{ij}$ defined in \eqref{eq:gij}
is positively $0$-homogeneous on each $T_xM$, and hence
\begin{equation}\label{eq:Av}
\sum_{i=1}^n A_{ijk}(v)v^i =\sum_{j=1}^n A_{ijk}(v)v^j
 =\sum_{k=1}^n A_{ijk}(v)v^k =0
\end{equation}
for all $v \in TM \setminus 0$ and $i,j,k=1,2,\ldots,n$,
where
\[ A_{ijk}(v):=\frac{F(v)}{2} \frac{\del g_{ij}}{\del v^k}(v),
 \qquad v \in TM \setminus 0, \]
is the \emph{Cartan tensor} which
measures the variation of $g_v$ in the vertical directions.
The Cartan tensor vanishes everywhere on $TM \setminus 0$
if and only if $F$ comes from a Riemannian metric.
In this sense, the Cartan tensor is a genuinely \emph{non-Riemannian} quantity.

With the help of \eqref{eq:Av},
we arrive at the following geodesic equation
by the usual calculation similar to the Riemannian case:
\begin{equation}\label{eq:geod}
\ddot{\eta}^i(t) +\sum_{j,k=1}^n \gamma^i_{jk} \big( \dot{\eta}(t) \big)
 \dot{\eta}^j(t) \dot{\eta}^k(t) =0
\end{equation}
for all $i$, where
\begin{equation}\label{eq:gamma}
\gamma^i_{jk}(v):=\frac{1}{2}\sum_{l=1}^n g^{il}(v) \bigg\{
 \frac{\del g_{lk}}{\del x^j}(v) +\frac{\del g_{jl}}{\del x^k}(v)
 -\frac{\del g_{jk}}{\del x^l}(v) \bigg\},
 \quad v \in TM \setminus 0,
\end{equation}
is called the \emph{formal Christoffel symbol}.
We denoted by $(g^{ij}(v))$ the inverse matrix of $(g_{ij}(v))$.
Notice that $\gamma^i_{jk}$ has the same form as the Riemannian Christoffel symbol,
while it is a ($0$-homogeneous) function on $TM \setminus 0$
and cannot be reduced to a function on $M$.

By the general ODE theory, every initial vector $v \in TM$ admits
a unique geodesic $\eta:(-\ve,\ve) \lra M$ with $\dot{\eta}(0)=v$ for some $\ve>0$.
Given $v \in T_xM$, if there is a geodesic $\eta:[0,1] \lra M$ with $\dot{\eta}(0)=v$,
then we define the \emph{exponential map} by $\exp_x(v):=\eta(1)$.
We say that $(M,F)$ is \emph{forward complete} if the exponential
map is defined on whole $TM$.
In other words, every geodesic $\eta:[0,1] \lra M$ is extended
infinitely in the forward direction to the geodesic $\bar{\eta}:[0,\infty) \lra M$.
If every geodesic $\eta:[0,1] \lra M$ is extended in the backward direction to the geodesic
$\bar{\eta}:(-\infty,1] \lra M$ (in other words, if $(M,\rev{F})$ is forward complete),
then we say that $(M,F)$ is \emph{backward complete}.
The backward completeness is not necessarily equivalent to
the forward completeness (in the noncompact case).
If $(M,F)$ is forward or backward complete,
then the Hopf--Rinow theorem ensures that any pair of points
is connected by a minimal geodesic (see \cite[Theorem~6.6.1]{BCS}).

\subsection{Covariant derivative}\label{ssc:covd}

We saw in \eqref{eq:geod} and \eqref{eq:gamma} that the Finsler geodesic equation
has a similar form to the Riemannian one.
This is, however, a special feature of the geodesic equation
and we need to take care of some non-Riemannian quantities in
the more general covariant derivative.
A fine property of the geodesic equation could be understood from Theorem~\ref{th:Euler}.
In order to apply Theorem~\ref{th:Euler} to some quantity at $v \in TM \setminus 0$,
we need a contraction with respect to $v$.
In the geodesic equation \eqref{eq:geod} we have the contractions twice thanks to
$\dot{\eta}^j(t)$ and $\dot{\eta}^k(t)$, this procedure kills all the error (non-Riemannian) terms.

Notice that the geodesic equation for the Riemannian structure $g_{\dot{\eta}}$
(defined only along $\eta$) coincides with the geodesic equation \eqref{eq:geod} with respect to $F$.
This kind of property is extremely useful when we try to apply the techniques in
Riemannian geometry to Finsler geometry.
This viewpoint leads us to modify the formal Christoffel symbol \eqref{eq:gamma} into
\begin{equation}\label{eq:Gamma}
\Gamma^i_{jk}(v):=\gamma^i_{jk}(v)
 -\sum_{l,m=1}^n \frac{g^{il}}{F}(A_{lkm}N^m_j +A_{jlm}N^m_k -A_{jkm}N^m_l)(v),
 \quad v \in TM \setminus 0,
\end{equation}
where
\[ N^i_j(v):=\frac{1}{2} \frac{\del G^i}{\del v^j}(v), \qquad
 G^i(v):=\sum_{j,k=1}^n \gamma^i_{jk}(v) v^j v^k \]
(the validity of the definition \eqref{eq:Gamma} can be seen in Proposition~\ref{pr:covd} below).
We call $G^i$ and $N^i_j$ the \emph{geodesic spray coefficients}
and the \emph{nonlinear connection}, respectively,
and set $G^i(0)=N^i_j(0):=0$ by convention.
Notice that \eqref{eq:Av} yields
\[ \sum_{j,k=1}^n \Gamma^i_{jk}(v) v^j v^k
 =\sum_{j,k=1}^n \gamma^i_{jk}(v) v^j v^k. \]
However, when we contract in $v$ only once,
\[ \sum_{j=1}^n \{ \Gamma^i_{jk}(v) -\gamma^i_{jk}(v) \} v^j
 =-\sum_{l,m=1}^n \frac{g^{il}(v)}{F(v)} A_{lkm}(v) G^m(v) \]
does not necessarily vanish.

Define the \emph{covariant derivative} of a vector field $X$ by $v \in T_xM$
with the \emph{reference vector} $w \in T_xM \setminus 0$ as
\begin{equation}\label{eq:covd}
D_v^w X(x):=\sum_{i,j=1}^n \bigg\{ v^j \frac{\del X^i}{\del x^j}(x)
 +\sum_{k=1}^n \Gamma^i_{jk}(w) v^j X^k(x) \bigg\} \frac{\del}{\del x^i}\Big|_x \in T_xM.
\end{equation}
With this definition, we can show the following important property
(see \cite[\S 6.2]{Shlec} or \cite[Lemma~2.3]{Osplit}).

\begin{proposition}\label{pr:covd}
Let $V$ be a non-vanishing $\cC^{\infty}$-vector field
on an open set $U$ such that all integral curves are geodesic.
Then we have, for any $\cC^1$-vector field $W$ on $U$,
\[ D_V^V W=D_V^{g_V} W, \qquad D_W^V V=D_W^{g_V} V. \]
Here $g_V$ is the Riemannian structure of $U$ induced from $V$ as \eqref{eq:gv}
and $D^{g_V}$ denotes its corresponding covariant derivative.
\end{proposition}

In particular, integral curves of $V$ are geodesic also with respect to $g_V$.
Note that we have only one contraction (with respect to $V$) in $D_V^V W$ and $D_W^V V$.
The geodesic equation $D^{\dot{\eta}}_{\dot{\eta}} \dot{\eta} \equiv 0$
is concerned with the special case of $W=V$, where we have one more contraction.

\subsection{Curvatures}\label{ssc:curv}

In order to define the curvature, we again look at the behavior of geodesics.
It can be shown that the variational vector field $J$
of a geodesic variation $\sigma:[0,1] \times (-\ve,\ve) \lra M$
(i.e., $\sigma(\cdot,s)$ is geodesic for all $s$ and $J(t):=\del \sigma/\del s(t,0)$)
satisfies the \emph{Jacobi equation}
\begin{equation}\label{eq:Jacobi}
D_{\dot{\eta}}^{\dot{\eta}} D_{\dot{\eta}}^{\dot{\eta}} J +R_{\dot{\eta}}(J)=0
\end{equation}
and vice versa, where $\eta(t):=\sigma(t,0)$ and
\begin{align*}
R_v(w) &:= \sum_{i,j=1}^n R^i_j(v) w^j \frac{\del}{\del x^i} \Big|_x, \\
R^i_j(v) &:= \frac{\del G^i}{\del x^j}(v)
 -\sum_{k=1}^n \bigg\{ \frac{\del N^i_j}{\del x^k}(v) v^k
 -\frac{\del N^i_j}{\del v^k}(v) G^k(v) \bigg\}
 -\sum_{k=1}^n N^i_k(v) N^k_j(v).
\end{align*}
We refer to \cite[\S 6.1]{Shlec} for details (where $G^i$ is one-half of ours,
while $N^i_j$ is the same).
It is unnecessary to worry about the complicated formula of $R^i_j(v)$,
what we essentially need will be only the Jacobi equation \eqref{eq:Jacobi}
and the characterization in Theorem~\ref{th:curv} below.

\begin{definition}[Flag curvature]\label{df:flag}
For linearly independent vectors $v,w \in T_xM$,
we define the \emph{flag curvature} by
\[ \bK(v,w):=\frac{g_v(R_v(w),w)}{F^2(v) g_v(w,w) -g_v(v,w)^2}. \]
\end{definition}

On Riemannian manifolds, the flag curvature $\bK(v,w)$ coincides with
the sectional curvature of the $2$-plane $v \wedge w$ spanned by $v$ and $w$.
We remark that $\bK(v,w)$ depends not only on the \emph{flag} $v \wedge w$,
but also on the choice of the \emph{pole} $v$ in it.
Thus, for example, $\bK(w,v)$ may be different from $\bK(v,w)$.
We further define the Ricci curvature as follows.

\begin{definition}[Ricci curvature]\label{df:Ric}
For a unit vector $v \in U_xM$,
we define the \emph{Ricci curvature} as the trace of the flag curvature
with respect to $g_v$, namely
\[ \Ric(v) :=\sum_{i=1}^{n-1} \bK(v,e_i), \]
where $\{e_i\}_{i=1}^{n-1} \cup \{v\}$ is orthonormal with respect to $g_v$.
We also define $\Ric(cv):=c^2 \Ric(v)$ for $c \ge 0$.
\end{definition}

As usual, given $K \in \R$,
the bound $\Ric \ge K$ will mean that $\Ric(v) \ge KF^2(v)$ holds for all $v \in TM$.
We remark that this curvature bound is common to $F$ and $\rev{F}$
since one can see $\rev{\Ric}(v)=\Ric(-v)$ from Theorem~\ref{th:curv}
(recall Definition~\ref{df:rev} for the definitions of $\rev{F}$ and $\rev{\Ric}$).

\begin{remark}\label{rm:scalar}
The appropriate notion of \emph{scalar curvature} is still missing in the Finsler setting.
This is one of the main obstructions to develop the theory of \emph{Finsler--Ricci flow}.
\end{remark}

Comparing Proposition~\ref{pr:covd} and \eqref{eq:Jacobi},
one finds the following quite important and useful property.

\begin{theorem}[A Riemannian characterization of Finsler curvatures]\label{th:curv}
Given a unit vector $v \in U_xM$, take a $\cC^{\infty}$-vector field $V$
on a neighborhood $U$ of $x$ such that $V(x)=x$ and all integral curves of $V$ are geodesic.
Then, for any $w \in T_xM$ linearly independent from $v$,
the flag curvature $\bK(v,w)$ coincides with the sectional curvature
of the $2$-plane $v \wedge w$ with respect to the Riemannian metric $g_V$ on $U$.
Similarly, $\Ric(v)$ coincides with the Ricci curvature of $v$ with respect to $g_V$.
\end{theorem}

We in particular find that the sectional curvature of $v \wedge w$
with respect to $g_V$ does not depend on the choice of $V$,
where the assumption on $V$ (all integral curves are geodesic) is playing the essential role.
The proof of Theorem~\ref{th:curv} can be found in \cite[\S 6]{Shlec},
which inspired the structure of this section.
The characterization in the manner of Theorem~\ref{th:curv} goes back to at least \cite{Au}.

Via Theorem~\ref{th:curv}, one can \emph{reduce} some results
on Finsler manifolds to the Riemannian case.
The most fundamental ones are the following (obtained in \cite{Au}),
we will see some more (concerning the weighted Ricci curvature) in the next subsection.

\begin{theorem}[Cartan--Hadamard theorem]\label{th:CH}
Let $(M,F)$ be a simply-connected and forward complete Finsler manifold
with nonpositive flag curvature $\bK \le 0$.
Then the exponential map $\exp_x:T_xM \lra M$ is a $\cC^1$-diffeomorphism for every $x \in M$.
\end{theorem}

\begin{theorem}[Bonnet--Myers theorem]\label{th:BM}
If $(M,F)$ is forward complete and satisfies $\Ric \ge K$ for some $K>0$, then we have
\[ \diam(M) :=\sup_{x,y \in M} d(x,y) \le \pi \sqrt{\frac{n-1}{K}}. \]
In particular, $M$ is compact.
\end{theorem}

\begin{outline}
Let us outline the proof of Theorem~\ref{th:BM}
(Theorem~\ref{th:CH} is shown in a similar manner).
Fix a point $x \in M$ and consider unit speed geodesics $\eta:[0,l_{\eta}) \lra M$ emanating from $x$.
Here $l_{\eta} \in (0,\infty)$ is taken as the supremum of $t>0$ such that $d(x,\eta(t))=t$.
Define the vector field $V$ by $V(\eta(t)):=\dot{\eta}(t)$ for each $\eta$ and $t \in (0,l_{\eta})$.
Then $V$ is a $\cC^{\infty}$-vector field on an open set $U \subset M \setminus \{x\}$
and all integral curves are geodesic.
Thanks to Theorem~\ref{th:curv}, the Riemannian metric $g_V$ has the Ricci curvature $\ge K$.
Therefore the \emph{Riemannian} Bonnet--Myers theorem applies
and we have $l_{\eta} \le \pi\sqrt{(n-1)/K}$.
This shows the claim.
$\qedd$
\end{outline}

\subsection{Weighted Ricci curvature}\label{ssc:wRic}

In order to discuss further applications of the Ricci curvature,
it is natural to employ a suitable measure on $M$.
At this point, our Finsler setting has another difficulty that
the choice of a canonical measure is not unique.
There are several known ways to generalize the Riemannian volume measure,
canonical in their own rights, leading to the different measures.
Among others the most fundamental examples of constructive measures are
the \emph{Busemann--Hausdorff measure} and the \emph{Holmes--Thompson measure}
(see \cite{AT}).

Our standpoint is that, instead of starting from a constructive measure,
we consider an arbitrary measure (this turns out natural, see Remark~\ref{rm:S-curv}).
Then, inspired by Theorem~\ref{th:curv}
as well as the theory of weighted Ricci curvature
(also called the Bakry--\'Emery--Ricci curvature) of Riemannian manifolds,
one can define the \emph{weighted Ricci curvature} as in \cite{Oint}.
From here on, we fix a positive $\cC^{\infty}$-measure $\fm$ on $M$
(meaning that it is written down as $\fm=\Phi \,dx^1 \cdots dx^n$ in each local coordinates
with some positive $\cC^{\infty}$-function $\Phi$).

\begin{definition}[Weighted Ricci curvature]\label{df:wRic}
Given a unit vector $v \in U_xM$,
let $V$ be a non-vanishing $\cC^{\infty}$-vector field
on a neighborhood $U$ of $x$ such that $V(x)=v$ and
all integral curves of $V$ are geodesic.
We decompose $\fm$ as $\fm=\e^{-\Psi}\vol_{g_V}$ on $U$,
where $\Psi \in \cC^{\infty}(U)$ and $\vol_{g_V}$ is the volume measure of $g_V$.
Denote by $\eta:(-\ve,\ve) \lra M$ the geodesic such that $\dot{\eta}(0)=v$.
Then, for $N \in (-\infty,0) \cup (n,\infty)$, define
\[ \Ric_N(v):=\Ric(v) +(\Psi \circ \eta)''(0) -\frac{(\Psi \circ \eta)'(0)^2}{N-n}. \]
We also define as the limits:
\[ \Ric_{\infty}(v):=\Ric(v) +(\Psi \circ \eta)''(0), \qquad
 \Ric_n(v):=\lim_{N \downarrow n}\Ric_N(v). \]
Finally, for $c \ge 0$, we set $\Ric_N(cv):=c^2 \Ric_N(v)$.
\end{definition}

We will denote by $\Ric_N \ge K$, $K \in \R$, the condition
$\Ric_N(v) \ge KF^2(v)$ for all $v \in TM$.
Some remarks on Definition~\ref{df:wRic} are in order.

\begin{remark}\label{rm:wRic}
(a)
In local coordinates, $\vol_{g_V}$ along $\eta$ is written as
\[ \sqrt{\det(g_{ij}\big( \dot{\eta}) \big)} \,dx^1 \cdots dx^n. \]
Hence we find
\[ \Psi(\eta) =\log \bigg( \frac{\sqrt{\det(g_{ij}(\dot{\eta}))}}{\Phi(\eta)} \bigg),
 \qquad \text{where}\ \fm =\Phi \,dx^1 \cdots dx^n. \]
This expression does not require the vector field $V$
(though we introduced it for the sake of lucidity),
thereby $\Psi \circ \eta$ does not depend on the choice of $V$.

(b)
On a Riemnnian manifold $(M,g)$,
the metric $g_V$ always coincides with the original metric $g$
and hence $\Psi$ is regarded as a function on $M$ such that $\fm=\e^{-\Psi} \vol_g$.
If we choose $\fm=\vol_g$, then $\Ric_N=\Ric$ regardless of the choice of $N$.

(c)
Multiplying the measure $\fm$ with a positive constant
does not change the weighted Ricci curvature.
Therefore we can freely normalize $\fm$ as $\fm(M)=1$ when $\fm(M)<\infty$.

(d)
Notice that $\Ric_N$ is non-decreasing in $N$ in the ranges $[n,\infty]$ and $(-\infty,0)$,
and we have
\[ \Ric_n \le \Ric_N \le \Ric_{\infty} \le \Ric_{N'} \qquad
 \text{for}\ N \in (n,\infty),\ N'<0. \]
Therefore, $\Ric_{N'} \ge K$ with $N'<0$ is a weaker condition
than $\Ric_N \ge K$ for $N \in (n,\infty)$.
In the Riemannian case, the study of $\Ric_{\infty}$ goes back to Lichnerowicz \cite{Li},
he showed a Cheeger--Gromoll type splitting theorem
(see \cite{FLZ,WW,Osplit} for some generalizations).
The range $N \in (n,\infty)$ has been well investigated by Bakry \cite[\S 6]{Ba},
Qian \cite{Qi} and many others.
The study of the range $N \in (-\infty,0)$ is more recent.
We refer to \cite{Mineg} for isoperimetric inequalities,
\cite{Oneg,Oneedle} for the curvature-dimension condition (with $N \le 0$),
and \cite{Wy} for splitting theorems (with $N \in (-\infty,1]$).
Some historical accounts on related works concerning $N<0$
in convex geometry and partial differential equations can be found in \cite{Mineg,Miharm}.
\end{remark}

Related to (b) above,
a reason why we consider it natural to begin with an arbitrary measure
is explained as follows.

\begin{remark}[$\bS$-curvature]\label{rm:S-curv}
For Finsler manifolds of \emph{Berwald type}
(i.e., $\Gamma_{ij}^k$ is constant on each $T_xM \setminus 0$),
the Busemann--Hausdorff measure satisfies $(\Psi \circ \eta)' \equiv 0$
(in other words, the \emph{$\bS$-curvature} vanishes, see \cite[\S 7.3]{Shlec}).
In general, however, there may not exist any measure whose $\bS$-curvature vanishes
(see \cite{ORand} for such an example among Randers spaces).
We regard this fact as the non-existence of a canonical measure,
thus we began with an arbitrary measure.
\end{remark}

The weighted version of the Bonnet--Myers theorem (Theorem~\ref{th:BM})
can be shown in the same manner as the unweighted case.

\begin{theorem}[Bonnet--Myers theorem]\label{th:wBM}
If $(M,F,\fm)$ is forward complete and
satisfies $\Ric_N \ge K$ for some $K>0$ and $N \in [n,\infty)$,
then we have
\[ \diam(M) \le \pi \sqrt{\frac{N-1}{K}}. \]
In particular, $M$ is compact.
\end{theorem}

From this estimate $N$ can be regarded as (an upper bound of) the dimension
(though this interpretation prevents us from considering $N<0$).
We can further control the volume growth for $\fm$.
For $x \in M$ and $r>0$, we define
\[ B^+(x,r):=\{ y \in M \,|\, d(x,y)<r \}. \]
We also introduce the function $\bs_{\kappa}$ defined by
\[ \bs_{\kappa}(\theta) := \left\{
 \begin{array}{cl}
 \frac{1}{\sqrt{\kappa}} \sin(\sqrt{\kappa}\theta) & \text{if}\ \kappa>0, \\
 \theta & \text{if}\ \kappa=0, \\
 \frac{1}{\sqrt{-\kappa}} \sinh(\sqrt{-\kappa}\theta) & \text{if}\ \kappa<0,
 \end{array} \right. \]
where $\theta \ge 0$ if $\kappa \le 0$ and
$\theta \in [0,\pi/\sqrt{\kappa}]$ if $\kappa>0$.

\begin{theorem}[Bishop--Gromov volume comparison]\label{th:BG}
Suppose that $(M,F,\fm)$ is \linebreak
forward complete and satisfies $\Ric_N \ge K$
for some $K \in \R$ and $N \in [n,\infty)$.
Then we have
\[ \frac{\fm(B^+(x,R))}{\fm(B^+(x,r))} \le
 \frac{\int_0^R \bs_{K/(N-1)}(t)^{N-1} \,dt}{\int_0^r \bs_{K/(N-1)}(t)^{N-1} \,dt} \]
for any $x \in M$ and $0<r<R$,
where $R \le \pi\sqrt{(N-1)/K}$ if $K>0$.
\end{theorem}

Note that the condition $R \le \pi\sqrt{(N-1)/K}$ for $K>0$
does not lose any generality because of Theorem~\ref{th:wBM}.
Similarly to Theorems~\ref{th:CH}, \ref{th:BM}, \ref{th:wBM},
the proof of Theorem~\ref{th:BG} can be reduced to the (weighted) Riemannian situation
by employing the Riemannian metric induced from the unit speed geodesics
emanating from $x$.

\begin{remark}[Curvature-dimension condition]\label{rm:CD}
The original motivation of introducing $\Ric_N$ in \cite{Oint}
was to study Lott, Sturm and Villani's \emph{curvature-dimension condition} $\CD(K,N)$
on Finsler manifolds (see also the surveys \cite{Osurv,Ogren}).
The celebrated condition $\CD(K,N)$ was introduced as a synthetic geometric notion
playing a role of a lower Ricci curvature bound for metric measure spaces (\cite{StI,StII,LV}).
On a complete Riemannian manifold with an arbitrary measure $(M,g,\fm)$,
the condition $\CD(K,N)$ is indeed equivalent to $\Ric_N \ge K$
(see \cite{vRS,StI,StII,LV}, where $N \in [n,\infty]$).
In \cite{Oint} this equivalence was generalized to Finsler manifolds $(M,F,\fm)$,
with applications including the Bishop--Gromov volume comparison
and some functional inequalities (see also Section~\ref{sc:funct} for the latter).
As we have already mentioned in Remark~\ref{rm:wRic}(d),
the curvature-dimension condition can be extended to $N \in (-\infty,0]$
and is again equivalent to $\Ric_N \ge K$ (\cite{Oneg,Oneedle}).

The theory of metric measure spaces satisfying $\CD(K,N)$
is making a deep and breathtaking progress in this decade.
We refer to Villani's book \cite{Vi} as a fundamental reference,
and Cavalletti--Mondino's L\'evy--Gromov type isoperimetric inequalities
(\cite{CMisop}, see also Section~\ref{sc:isop})
as one of the most notable achievements after \cite{Vi}.
Since the condition $\CD(K,N)$ could not rule out Finsler manifolds,
a reinforced version of $\CD(K,N)$ coupled with the linearity of the heat semigroup
was introduced and is called the \emph{Riemannian curvature-dimension condition}
$\RCD(K,N)$ (\cite{AGSrcd,EKS}).
In this Riemannian framework, many finer results such as
Gigli's splitting theorem \cite{G-split} and second order calculus \cite{G-nsdg}
as well as Mondino--Naber's rectifiability \cite{MN} are known.
\end{remark}

\section{Nonlinear Laplacian and heat semigroup}\label{sc:heat}

At this section we begin geometric analysis on Finsler manifolds.
The main object of the section is the Bochner--Weitzenb\"ock formula (Theorem~\ref{th:BW}),
which is the indispensable ingredient of our nonlinear analogue of the $\Gamma$-calculus.

\subsection{Gradient vectors, Hessian and Laplacian}\label{ssc:Lap}

Let us denote by $\cL^*:T^*M \lra TM$ the \emph{Legendre transform}
associated with $F$.
That is to say, $\cL^*$ is sending $\alpha \in T_x^*M$ to the unique element $v \in T_xM$
such that $F(v)=F^*(\alpha)$ and $\alpha(v)=F^*(\alpha)^2$.
In coordinates we can write down
\[ \cL^*(\alpha)=\sum_{i,j=1}^n g_{ij}^*(\alpha) \alpha_j \frac{\del}{\del x^i}\Big|_x
 =\sum_{i=1}^n \frac{1}{2} \frac{\del[(F^*)^2]}{\del \alpha_i}(\alpha) \frac{\del}{\del x^i}\Big|_x \]
for $\alpha \in T_x^*M \setminus 0$
(the latter expression makes sense also at $0$ as $\cL^*(0)=0$).
Note that
\[ g^*_{ij}(\alpha) =g^{ij}\big( \cL^*(\alpha) \big)
 \qquad \text{for}\ \alpha \in T_x^*M \setminus 0 \]
(recall that $(g^{ij}(v))$ denotes the inverse matrix of $(g_{ij}(v))$).
The map $\cL^*|_{T^*_xM}$ is being a linear operator only when $F|_{T_xM}$
comes from an inner product.

For a differentiable function $u:M \lra \R$, the \emph{gradient vector}
at $x$ is defined as the Legendre transform of the derivative of $u$:
$\Grad u(x):=\cL^*(Du(x)) \in T_xM$.
If $Du(x) \neq 0$, then we can write down in coordinates as
\[ \Grad u(x)
 =\sum_{i,j=1}^n g^*_{ij}\big( Du(x) \big)
 \frac{\del u}{\del x^j}(x) \frac{\del}{\del x^i}\Big|_x. \]
We need to be careful when $Du(x)=0$,
because $g^*_{ij}(Du(x))$ is not defined as well as
the Legendre transform $\cL^*$ is only continuous at the zero section.
Thus we set for later use
\[ M_u:=\{ x \in M \,|\, Du(x) \neq 0 \}. \]
For a twice differentiable function $u:M \lra \R$ and $x \in M_u$,
we define a kind of \emph{Hessian} $\Grad^2 u(x) :T_xM \lra T_xM$
by using the covariant derivative \eqref{eq:covd} as
\[ \Grad^2 u(v):=D^{\Grad u}_v (\Grad u)(x) \in T_xM, \qquad v \in T_xM. \]
The linear operator $\Grad^2 u(x)$ is symmetric in the sense that
\[ g_{\Grad u}\big( \Grad^2 u(v),w \big)=g_{\Grad u}\big( v,\Grad^2 u(w) \big) \]
for all $v,w \in T_xM$ with $x \in M_u$
(see, for example, \cite[Lemma~2.3]{OSbw}).

Define the \emph{divergence} of a differentiable vector field $V$ on $M$
with respect to the measure $\fm$ by
\[ \div_{\fm} V:=\sum_{i=1}^n \bigg( \frac{\del V^i}{\del x^i} +V^i \frac{\del \Phi}{\del x^i} \bigg),
 \qquad V=\sum_{i=1}^n V^i \frac{\del}{\del x^i}, \]
where we decomposed $\fm$ as $d\fm=\e^{\Phi} \,dx^1 dx^2 \cdots dx^n$.
One can rewrite this in the weak form as
\[ \int_M \phi \div_{\fm} V \,d\fm =-\int_M D\phi(V) \,d\fm \qquad
 \text{for all}\ \phi \in \cC_c^{\infty}(M), \]
that makes sense for measurable vector fields $V$ with $F(V) \in L_{\loc}^1(M)$.
Then we define the distributional \emph{Laplacian} of $u \in H^1_{\loc}(M)$ by
$\Lap u:=\div_{\fm}(\Grad u)$ in the weak sense that
\[ \int_M \phi\Lap u \,d\fm:=-\int_M D\phi(\Grad u) \,d\fm \qquad
 \text{for all}\ \phi \in \cC_c^{\infty}(M). \]
Notice that the space $H^1_{\loc}(M)$ is defined solely in terms of the differentiable structure of $M$.
Since taking the gradient vector (more precisely, the Legendre transform $\cL^*$)
is a nonlinear operation, our Laplacian $\Lap$ is \emph{nonlinear} unless $F$ is Riemannian.

\begin{remark}\label{rm:Lap}
(a)
In the Riemannian case, the Laplacian $\Lap_{\fm}$ associated with
a measure $\fm=\e^{\Phi} \vol_g$ can be written as
\[ \Lap_{\fm} u=\Delta^g u +g(\Grad u,\Grad \Phi), \]
where $\Delta^g$ is the usual Laplacian with respect to $g$.
Then we call $\Lap_{\fm}$ the \emph{Witten Laplacian} or the \emph{weighted Laplacian}.

(b)
It is also possible to define the Laplacian associated only with $F$,
and regard our Laplacian $\Lap$ as the weighted one with respect to $\fm$.
Such a definition can be found in \cite{Lee} for instance (see also \cite{Oham}),
however, it is more involved than our simple definition $\Lap u=\div_{\fm}(\Grad u)$
that is also natural from the analytic viewpoint.
\end{remark}

\subsection{Bochner--Weitzenb\"ock formula}\label{ssc:BW}

Concerning the relation between the Laplacian and the Ricci curvature,
it is not difficult to show the \emph{Laplacian comparison}
(by essentially reducing to the Riemannian situation like Theorem~\ref{th:BG}, see \cite{OShf}).
This is regarded as an analytic counterpart of the directed version of the Bishop--Gromov volume comparison
(known as the \emph{measure contraction property}, see \cite{Omcp,StII}).

\begin{theorem}[Laplacian comparison]\label{th:Lap}
Let $(M,F)$ be forward complete,
and assume that $\Ric_N \ge 0$ for some $N \in [n,\infty)$.
Then, for any $z \in M$, the function $u(x):=d(z,x)$ satisfies
\[ \Lap u(x) \le \frac{N-1}{d(z,x)} \]
point-wise on $M \setminus (\{z\} \cup \Cut_z)$,
and in the distributional sense on $M \setminus \{z\}$.
\end{theorem}

Denoted by $\Cut_z$ is the \emph{cut locus} of $z$,
which is the set of \emph{cut points} $x=\exp_z(v)$ such that
$\eta(t)=\exp_z(tv)$ is minimal on $[0,1]$ but not minimal on $[0,1+\ve]$ for any $\ve>0$.

A more sophisticated application of our geometric analysis
is the \emph{Bochner--Weitzenb\"ock formula}.
This yields the \emph{Bochner inequality}, the heart of the $\Gamma$-calculus.
The Finsler versions of them were established in \cite{OSbw} (for $N \in [n,\infty]$)
and \cite{Oisop} (for $N<0$).
In order to state the formula, we need the following notations.
Given $f \in H^1_{\loc}(M)$ and a measurable vector field $V$
such that $V \neq 0$ almost everywhere on $M_f$,
we can define the gradient vector field and the Laplacian
on the weighted Riemannian manifold $(M,g_V,\fm)$ by
\[ \nabla^V f:=\left\{ \begin{array}{ll}
 \displaystyle\sum_{i,j=1}^n g^{ij}(V) \frac{\del f}{\del x^j} \frac{\del}{\del x^i}
 & \text{on}\ M_f, \smallskip\\
 0 & \text{on}\ M \setminus M_f, \end{array}\right.
 \qquad \Delta\!^V f:=\div_{\fm}(\nabla^V f), \]
where the latter is in the sense of distributions.
We have $\nabla^{\Grad u}u=\Grad u$ and $\Delta\!^{\Grad u}u=\Lap u$
for $u \in H^1_{\loc}(M)$ (\cite[Lemma~2.4]{OShf}).
We also observe that, given $u, f_1,f_2 \in H^1_{\loc}(M)$,
\begin{equation}\label{eq:f1f2}
Df_2(\nabla^{\Grad u}f_1) =g_{\Grad u}(\nabla^{\Grad u}f_1,\nabla^{\Grad u}f_2)
 =Df_1(\nabla^{\Grad u}f_2)
\end{equation}
(to be precise, in \eqref{eq:f1f2} we replace $\Grad u$ with a non-vanishing vector field
in a measurable way on $M \setminus M_u$).

\begin{theorem}[Bochner--Weitzenb\"ock formula]\label{th:BW}
Given $u \in \cC^{\infty}(M)$, we have
\begin{equation}\label{eq:BW}
\Delta\!^{\Grad u} \bigg[ \frac{F^2(\Grad u)}{2} \bigg] -D[\Lap u](\Grad u)
 =\Ric_{\infty}(\Grad u) +\| \Grad^2 u \|_{\HS(\Grad u)}^2
\end{equation}
as well as
\[ \Delta\!^{\Grad u} \bigg[ \frac{F^2(\Grad u)}{2} \bigg] -D[\Lap u](\Grad u)
 \ge \Ric_N(\Grad u) +\frac{(\Lap u)^2}{N} \]
for $N \in (-\infty,0) \cup [n,\infty]$ point-wise on $M_u$,
where $\|\cdot\|_{\HS(\Grad u)}$ denotes
the Hilbert--Schmidt norm with respect to $g_{\Grad u}$.
\end{theorem}

In particular, if $\Ric_N \ge K$, then we have
\begin{equation}\label{eq:Boc}
\Delta\!^{\Grad u} \bigg[ \frac{F^2(\Grad u)}{2} \bigg] -D[\Lap u](\Grad u)
 \ge KF^2(\Grad u) +\frac{(\Lap u)^2}{N}
\end{equation}
on $M_u$, that we will call the \emph{Bochner inequality}.
One can further generalize the Bochner--Weitzenb\"ock formula
to a more general class of Hamiltonian systems
(by dropping the positive $1$-homogeneity of $F$, see \cite{Lee,Oham}).

\begin{remark}[$F$ versus $g_{\Grad u}$]\label{rm:g_u}
We stress that \eqref{eq:BW} cannot be reduced to the Bochner--Weitzenb\"ock formula
of the Riemannian metric $g_{\Grad u}$.
In fact, in contrast to the identity $\Delta\!^{\Grad u}u=\Lap u$,
$\Ric_N(\Grad u)$ does not coincide with the weighted Ricci curvature
$\Ric_N^{\Grad u}(\Grad u)$ of the weighted Riemannian manifold $(M,g_{\Grad u},\fm)$.
It is compensated in \eqref{eq:BW} by the fact that $\Grad^2 u$ does not
necessarily coincide with the Hessian of $u$ with respect to $g_{\Grad u}$
(unless all integral curves of $\Grad u$ are geodesic).
\end{remark}

Recall that, even when $u \in \cC^{\infty}(M)$,
$\Grad u$ is only continuous outside $M_u$.
Thus we have to pass to the integrated form, and it is done as follows
(\cite[Theorem~3.6]{OSbw}, \cite{Oisop}).

\begin{theorem}[Integrated form]\label{th:BWint}
Assume $\Ric_N \ge K$ for some $K \in \R$ and $N \in (-\infty,0) \cup [n,\infty]$.
Given $u \in H_0^1(M) \cap H^2_{\loc}(M) \cap \cC^1(M)$
such that $\Lap u \in H_0^1(M)$,
we have
\begin{align}
&-\int_M D\phi \bigg( \nabla^{\Grad u} \bigg[ \frac{F^2(\Grad u)}{2} \bigg] \bigg) \,d\fm
 \nonumber\\
&\ge \int_M \phi \bigg\{ D[\Lap u](\Grad u) +K F^2(\Grad u)
 +\frac{(\Lap u)^2}{N} \bigg\} \,d\fm \label{eq:BWint}
\end{align}
for all nonnegative functions $\phi \in H_{\loc}^1(M) \cap L^{\infty}(M)$.
\end{theorem}

See the next subsection for the definition of $H^1_0(M)$.
We see in Theorem~\ref{th:hf}(ii) below that global solutions $(u_t)_{t \ge 0}$
to the heat equation always enjoy
the condition $u_t \in H_0^1(M) \cap H^2_{\loc}(M) \cap \cC^1(M)$,
and also $\Lap u_t \in H_0^1(M)$ when $\sS_F <\infty$.

\subsection{Heat equation}\label{ssc:hf}

In \cite{OShf,OSbw}, we have investigated the \emph{nonlinear heat equation}
$\del_t u=\Lap u$ associated with our Laplacian $\Lap$.
To recall some results in \cite{OShf},
we define the \emph{energy} of $u \in H_{\loc}^1(M)$ by
\[ \cE(u):=\frac{1}{2}\int_M F^2(\Grad u) \,d\fm
 =\frac{1}{2}\int_M F^*(Du)^2 \,d\fm. \]
We remark that $\cE(u)<\infty$ does not necessarily imply $\cE(-u)<\infty$.
Define $H^1_0(M)$ as the closure of $\cC_c^{\infty}(M)$
with respect to the (absolutely homogeneous) norm
\[ \|u\|_{H^1}:=\|u\|_{L^2} +\{ \cE(u)+\cE(-u) \}^{1/2}. \]
Note that $(H^1_0(M),\|\cdot\|_{H^1})$ is a Banach space.

\begin{definition}[Global solutions]\label{df:hf}
We say that a function $u$ on $[0,T] \times M$, $T>0$,
is a \emph{global solution} to the heat equation $\del_t u=\Lap u$ if
it satisfies the following$:$
\begin{enumerate}[(1)]
\item $u \in L^2\big( [0,T],H^1_0(M) \big) \cap H^1 \big( [0,T],H^{-1}(M) \big)$;

\item We have
\[ \int_M \phi \cdot \del_t u_t \,d\fm =-\int_M D\phi(\Grad u_t) \,d\fm \]
for all $t \in [0,T]$ and $\phi \in \cC_c^{\infty}(M)$, where we set $u_t:=u(t,\cdot)$.
\end{enumerate}
\end{definition}

We refer to \cite{Ev} for the notations as in (1).
The test function $\phi$ in (2) can be taken from $H^1_0(M)$.
Global solutions are constructed as gradient curves of the energy functional $\cE$
in the Hilbert space $L^2(M)$.
As for the regularity, since our Laplacian is locally uniformly elliptic
thanks to the strong convexity of $F$,
the classical theory of partial differential equations applies.
We summarize the existence and regularity properties in the next theorem
(see \cite[\S\S 3, 4]{OShf} for details,
we remark that our $\cC^{\infty}$-smooth $F$ and $\fm$ clearly enjoy
the mild smoothness assumption in \cite[(4.4)]{OShf}).

\begin{theorem}\label{th:hf}
\begin{enumerate}[{\rm (i)}]
\item
For each initial datum $u_0 \in H^1_0(M)$ and $T>0$,
there exists a unique global solution $u=(u_t)_{t \in [0,T]}$ to the heat equation,
and the distributional Laplacian $\Lap u_t$ is absolutely continuous
with respect to $\fm$ for all $t \in (0,T)$.

\item
One can take the continuous version of a global solution $u$, and it enjoys
the $H^2_{\loc}$-regularity in $x$ as well as the $\cC^{1,\alpha}$-regularity in both $t$ and $x$.
Moreover, $\del_t u$ lies in $H^1_{\loc}(M) \cap \cC(M)$,
and further in $H_0^1(M)$ if $\sS_F <\infty$.
\end{enumerate}
\end{theorem}

We remark that the usual elliptic regularity yields that
$u$ is $\cC^{\infty}$ on $\bigcup_{t>0}(\{t\} \times M_{u_t})$.
The proof of $\del_t u \in H_0^1(M)$ under $\sS_F <\infty$
can be found in \cite[Appendix~A]{OShf}.
It was also proved in \cite{OShf} that the heat flow is regarded as the gradient flow
of the relative entropy (see the proof of Proposition~\ref{pr:5.7.3} below)
in the \emph{$L^2$-Wasserstein space} with respect to $\rev{F}$
(this result is far beyond the scope of the present survey).

\section{Linearized heat semigroups and gradient estimates}\label{sc:grades}

In the Bochner--Weitzenb\"ock formula (Theorem~\ref{th:BW}) in the previous section,
we used the linearized Laplacian $\Delta\!^{\Grad u}$ induced from
the Riemannian structure $g_{\Grad u}$.
In the same spirit, we can consider the linearized heat semigroup
associated with a global solution to the heat equation.
This technique turns out useful, and we obtain some gradient estimates
as our first applications of the nonlinear $\Gamma$-calculus.
We will find that, from here on,
our arguments rely only on the Bochner inequality, the integration by parts,
and the (nonlinear and linearized) heat semigroups.
Calculations in local coordinates do not appear.

\subsection{Linearized heat semigroups and their adjoints}\label{ssc:lin}

Let $(u_t)_{t \ge 0}$ be a global solution to the heat equation.
We will fix a measurable one-parameter family of \emph{non-vanishing} vector fields
$(V_t)_{t \ge 0}$ such that $V_t=\Grad u_t$ on $M_{u_t}$ for each $t \ge 0$.

Given $f \in H^1_0(M)$ and $s \ge 0$,
let $(P_{s,t}^{\Grad u}(f))_{t \ge s}$ be the weak solution to the \emph{linearized heat equation}:
\begin{equation}\label{eq:lin-hf}
\del_t [P_{s,t}^{\Grad u}(f)] =\Delta\!^{V_t}[P_{s,t}^{\Grad u}(f)],
 \qquad P_{s,s}^{\Grad u}(f)=f.
\end{equation}
The existence and other properties of the linearized semigroup $P_{s,t}^{\Grad u}$
are summarized as follows (\cite{OShf,Oisop}).

\begin{proposition}[Properties of linearized semigroups]\label{pr:lin}
Assume $\sS_F<\infty$, and
let $(u_t)_{t \ge 0}$ and $(V_t)_{t \ge 0}$ be as above.
\begin{enumerate}[{\rm (i)}]
\item
For each $s \ge 0$, $T>0$ and $f \in H^1_0(M)$, there exists a unique solution
$f_t=P^{\Grad u}_{s,t}(f)$, $t \in [s,s+T]$, to \eqref{eq:lin-hf} in the weak sense that
\[ \int_s^{s+T} \int_M \del_t \phi_t \cdot f_t \,d\fm \,dt
 =\int_s^{s+T} \int_M D\phi_t (\nabla^{V_t} f_t) \,d\fm \,dt \]
for all $\phi \in \cC_c^{\infty}((s,s+T) \times M)$.

\item
The solution $(f_t)_{t \in [s,s+T]}$ is H\"older continuous on $(s,s+T) \times M$
as well as $H_{\loc}^2$ and $\cC^{1,\alpha}$ in $x$.
Moreover, we have $\del_t f_t \in H^1_0(M)$ for $t \in (s,s+T)$.
\end{enumerate}
\end{proposition}

The uniqueness in (i) follows from
$\del_t (\|f_t\|_{L^2}^2) =-4\cE^{V_t}(f_t) \le 0$,
where $\cE^{V_t}$ is the energy form with respect to $g_{V_t}$.
This also implies that
$P^{\Grad u}_{s,t}$ uniquely extends to the linear contraction semigroup acting on $L^2(M)$.
Notice also that
\[ f \in \cC^{\infty} \bigg( \bigcup_{s<t<s+T} (\{t\} \times M_{u_t}) \bigg). \]

The operator $P_{s,t}^{\Grad u}$ is linear but \emph{asymmetric}
with respect to the $L^2$-inner product.
Let us denote by $\widehat{P}^{\Grad u}_{s,t}$ the \emph{adjoint operator}
of $P_{s,t}^{\Grad u}$.
That is to say, given $h \in H_0^1(M)$ and $t>0$,
we define $(\widehat{P}^{\Grad u}_{s,t}(h))_{s \in [0,t]}$ as the solution
to the equation
\begin{equation}\label{eq:Phat}
\del_s [\widehat{P}^{\Grad u}_{s,t}(h)]
 =-\Delta\!^{V_s} [\widehat{P}^{\Grad u}_{s,t}(h)],
 \qquad \widehat{P}^{\Grad u}_{t,t}(h)=h.
\end{equation}
Note that
\begin{equation}\label{eq:adj}
\int_M h \cdot P^{\Grad u}_{s,t}(f) \,d\fm
 =\int_M \widehat{P}^{\Grad u}_{s,t}(h) \cdot f \,d\fm
\end{equation}
indeed holds since for $r \in (0,t-s)$
\begin{align*}
&\del_r \bigg[ \int_M \widehat{P}^{\Grad u}_{s+r,t}(h) \cdot
 P_{s,s+r}^{\Grad u}(f) \,d\fm \bigg] \\
&= -\int_M \Delta\!^{V_{s+r}}[\widehat{P}^{\Grad u}_{s+r,t}(h)] \cdot
 P_{s,s+r}^{\Grad u}(f) \,d\fm
 +\int_M \widehat{P}^{\Grad u}_{s+r,t}(h) \cdot
 \Delta\!^{V_{s+r}}[P_{s,s+r}^{\Grad u}(f)] \,d\fm \\
&=0.
\end{align*}
One may rewrite \eqref{eq:Phat} as
\[ \del_{\sigma} [\widehat{P}^{\Grad u}_{t-\sigma,t}(h)]
 =\Delta\!^{V_{t-\sigma}} [\widehat{P}^{\Grad u}_{t-\sigma,t}(h)],
 \qquad \sigma \in [0,t], \]
to see that the adjoint heat semigroup solves the linearized heat equation \emph{backward in time}.
(This evolution is sometimes called the \emph{conjugate heat semigroup},
especially in the Ricci flow theory, see for instance \cite[Chapter~5]{Ch+}.)
Therefore we find in the same way as $P^{\Grad u}_{s,t}$ that
$\widehat{P}^{\Grad u}_{t-\sigma,t}$ extends to the linear contraction semigroup acting on $L^2(M)$.

\begin{remark}\label{rm:V_t}
In general, the semigroups $P^{\Grad u}_{s,t}$ and $\widehat{P}^{\Grad u}_{s,t}$
may depend on the choice of an auxiliary vector field $(V_t)_{t \ge 0}$.
This issue, however, will not affect our discussions.
\end{remark}

\subsection{Gradient estimates}\label{ssc:grades}

We deduce from the Bochner inequality \eqref{eq:Boc} with $N=\infty$
the \emph{$L^2$-gradient estimate} for heat flow.
The proof is a good example of a typical argument in the $\Gamma$-calculus.

\begin{theorem}[$L^2$-gradient estimate]\label{th:L2}
Assume $\Ric_{\infty} \ge K$ and $\sS_F<\infty$.
Then, given any global solution $(u_t)_{t \ge 0}$ to the heat equation
with $u_0 \in \cC^{\infty}_c(M)$, we have
\[ F^2\big( \Grad u_t(x) \big)
 \le \e^{-2K(t-s)} P_{s,t}^{\Grad u} \big( F^2(\Grad u_s) \big) (x) \]
for all $0 \le s<t<\infty$ and $x \in M$.
\end{theorem}

Let us stress that we used the nonlinear semigroup ($u_s \to u_t$) in the LHS,
while in the RHS the linearized semigroup $P^{\Grad u}_{s,t}$ shows up.

\begin{outline}
We remark that the condition $u_0 \in \cC^{\infty}_c(M)$ ensures
$F^2(\Grad u_0) \in \cC^1_c(M)$,
and hence both sides in Theorem~\ref{th:L2} are H\"older continuous.

For fixed $t>0$ and an arbitrary nonnegative function $h \in \cC(M)$ of compact support,
we set
\[ H(s):=\e^{2Ks} \int_M \widehat{P}_{s,t}^{\Grad u}(h) \cdot \frac{F^2(\Grad u_s)}{2} \,d\fm,
 \qquad 0 \le s \le t. \]
Then we deduce from the definition of $\widehat{P}_{s,t}^{\Grad u}$ that
\begin{align*}
H'(s) &=\e^{2Ks} \int_M D\bigg( \frac{F^2(\Grad u_s)}{2} \bigg)
 \Big( \nabla^{\Grad u_s} \big( \widehat{P}_{s,t}^{\Grad u}(h) \big) \Big) \,d\fm \\
&\quad +\e^{2Ks} \int_M \widehat{P}_{s,t}^{\Grad u}(h) \cdot
 \frac{\del}{\del s} \bigg( \frac{F^2(\Grad u_s)}{2} \bigg) \,d\fm
 +2KH(s).
\end{align*}
By \eqref{eq:f1f2}, the first term in the RHS coincides with
\[ \e^{2Ks} \int_M D \big( \widehat{P}_{s,t}^{\Grad u}(h) \big)
 \bigg( \nabla^{\Grad u_s} \bigg( \frac{F^2(\Grad u_s)}{2} \bigg) \bigg) \,d\fm. \]
The other terms are calculated with the help of Theorem~\ref{th:Euler} as
(see \cite{OSbw} for details)
\begin{align*}
&\e^{2Ks} \int_M \widehat{P}_{s,t}^{\Grad u}(h) \bigg\{
 D\bigg( \frac{\del u_s}{\del s} \bigg) (\Grad u_s) +KF^2(\Grad u_s) \bigg\} \,d\fm \\
&= \e^{2Ks} \int_M \widehat{P}_{s,t}^{\Grad u}(h) \big\{ D(\Lap u_s) (\Grad u_s)
 +KF^2(\Grad u_s) \big\} \,d\fm.
\end{align*}
Therefore we have $H'(s) \le 0$ by the Bochner inequality \eqref{eq:BWint}.

Thus we find $H(t) \le H(s)$ for any $0 \le s \le t$, which yields via \eqref{eq:adj}
\begin{align*}
\e^{2Kt} \int_M h \, \frac{F^2(\Grad u_t)}{2} \,d\fm
&\le \e^{2Ks} \int_M \widehat{P}_{s,t}^{\Grad u}(h) \cdot \frac{F^2(\Grad u_s)}{2} \,d\fm \\
&= \e^{2Ks} \int_M h \cdot P_{s,t}^{\Grad u} \bigg( \frac{F^2(\Grad u_s)}{2} \bigg) \,d\fm.
\end{align*}
Since this holds true for any nonnegative $h$, we obtain the claimed inequality
for almost all $x \in M$.
This completes the proof since both sides are H\"older continuous.
$\qedd$
\end{outline}

In the linear setting, it is known that the Bochner inequality
enjoys a self-improving property (\cite[Theorem~6]{BQ}, \cite{BGL}).
In our Finsler setting, however, the self-improvement works only when $F$ is reversible
(as far as the author knows).
Nonetheless, we can show the improved inequality by the direct calculation (\cite{Oisop}).

\begin{proposition}[Improved Bochner inequality]\label{pr:iBoch}
Assume $\Ric_{\infty} \ge K$ for some $K \in \R$.
Given $u \in H_0^1(M) \cap H^2_{\loc}(M) \cap \cC^1(M)$
such that $\Lap u \in H_0^1(M)$,
we have
\begin{align*}
&-\int_M D\phi \bigg( \nabla^{\Grad u} \bigg[ \frac{F^2(\Grad u)}{2} \bigg] \bigg) \,d\fm \\
&\ge \int_M \phi \Big\{ D[\Lap u](\Grad u) +K F^2(\Grad u)
 +D[F(\Grad u)]\big( \nabla^{\Grad u}[F(\Grad u)] \big) \Big\} \,d\fm
\end{align*}
for all nonnegative functions $\phi \in H_{\loc}^1(M) \cap L^{\infty}(M)$.
\end{proposition}

By using this improved inequality instead of \eqref{eq:BWint},
we have the \emph{$L^1$-gradient estimate}
(which implies the $L^2$-gradient estimate via Jensen's inequality, see \cite{Oisop}).

\begin{theorem}[$L^1$-gradient estimate]\label{th:L1}
Assume $\Ric_{\infty} \ge K$, $\sS_F<\infty$ and the completeness of $(M,F)$.
Then, given any global solution $(u_t)_{t \ge 0}$ to the heat equation
with $u_0 \in \cC^{\infty}_c(M)$, we have
\[ F\big( \Grad u_t (x) \big)
 \le \e^{-K(t-s)} P_{s,t}^{\Grad u} \big( F(\Grad u_s) \big)(x) \]
for all $0 \le s<t<\infty$ and $x \in M$.
\end{theorem}

Recall from Lemma~\ref{lm:rev} that
$\sS_F<\infty$ implies the finite reversibility $\Lambda_F<\infty$.
Hence $F$ and $\rev{F}$ are comparable
and the forward and backward completenesses are mutually equivalent.
Thereby we called it the plain \emph{completeness} in the theorem above.

\subsection{Characterizations of lower Ricci curvature bounds}\label{ssc:char}

We saw in the previous subsection that
the Bochner inequalities imply the gradient estimates.
The converse also holds true and, furthermore,
they characterize the lower Ricci curvature bound $\Ric_{\infty} \ge K$ as follows (\cite{Oisop}).

\begin{theorem}[Characterizations of Ricci curvature bounds]\label{th:char}
Suppose that $\sS_F<\infty$ and $(M,F)$ is complete.
Then, for each $K \in \R$, the following are equivalent$:$
\begin{enumerate}[{\rm (I)}]
\item $\Ric_{\infty} \ge K$.

\item The Bochner inequality
\[ \Delta\!^{\Grad u} \bigg[ \frac{F^2(\Grad u)}{2} \bigg] -D[\Lap u](\Grad u)
 \ge KF^2(\Grad u) \]
holds on $M_u$ for all $u \in \cC^{\infty}(M)$.

\item The improved Bochner inequality
\[ \Delta\!^{\Grad u} \bigg[ \frac{F^2(\Grad u)}{2} \bigg] -D[\Lap u](\Grad u) -KF^2(\Grad u)
 \ge D[F(\Grad u)] \big( \nabla^{\Grad u} [F(\Grad u)] \big) \]
holds on $M_u$ for all $u \in \cC^{\infty}(M)$.

\item The $L^2$-gradient estimate
\[ F^2(\Grad u_t)
 \le \e^{-2K(t-s)} P_{s,t}^{\Grad u} \big( F^2(\Grad u_s) \big), \qquad 0 \le s<t<\infty, \]
holds for all global solutions $(u_t)_{t \ge 0}$ to the heat equation
with $u_0 \in \cC^{\infty}_c(M)$.

\item The $L^1$-gradient estimate
\[ F(\Grad u_t)
 \le \e^{-K(t-s)} P_{s,t}^{\Grad u} \big( F(\Grad u_s) \big), \qquad 0 \le s<t<\infty, \]
holds for all global solutions $(u_t)_{t \ge 0}$ to the heat equation
with $u_0 \in \cC^{\infty}_c(M)$.
\end{enumerate}
\end{theorem}

In the Riemannian setting, one can add to this list
the \emph{contraction property} of heat flow with respect to
the \emph{$L^2$-Wasserstein distance} $W_2$
(see \cite{vRS,Ku} for details):
\[ W_2(u_t \fm,\bar{u}_t \fm) \le \e^{-Kt} W_2(u_0 \fm,\bar{u}_0 \fm) \]
for global solutions $(u_t)_{t \ge 0},(\bar{u}_t)_{t \ge 0}$
such that $u_0,\bar{u}_0 \ge 0$ as well as $\int_M u_0 \,d\fm=\int_M \bar{u}_0 \,d\fm=1$.
In our Finsler situation, however, the contraction property is known to fail (\cite{OSnc}).

\section{Functional inequalities}\label{sc:funct}

In this section we discuss three functional inequalities obtained in \cite{Ofunc}.
The proofs are based on the $\Gamma$-calculus,
we refer to the book \cite{BGL} for the usual linear setting.
In the reversible case, there is another technique called the \emph{localization}
which gives the same sharp estimates even in a more general framework
of essentially non-branching metric measure spaces
satisfying the curvature-dimension condition (see \cite{CMfunc}).
In the non-reversible case, however, the localization seems to give only
non-sharp estimates, see \S \ref{ssc:back} for a more detailed discussion.

\subsection{Poincar\'e--Lichnerowicz inequality}\label{ssc:Poin}

We start with the \emph{Poincar\'e--Lichnerowicz $($spectral gap$)$ inequality}
under the curvature bound $\Ric_N \ge K>0$.
The $N \in [n,\infty]$ case was shown in \cite{Oint}
as a consequence of the curvature-dimension condition $\CD(K,N)$.
In the Riemannian setting,
the case of $N \in (-\infty,0)$ was shown independently in \cite{KM} and \cite{Oneg}.

For simplicity we will assume that $M$ is compact
(this is automatically true when $N \in [n,\infty)$ and $M$ is complete),
and normalize $\fm$ as $\fm(M)=1$.
We will give the proofs since they are not long.

\begin{proposition}\label{pr:3.3.17}
Assume that $M$ is compact and satisfies
$\Ric_N \ge K$ for some $K \in \R$ and $N \in (-\infty,0) \cup [n,\infty]$.
Then we have, for any $f \in H^2(M) \cap \cC^1(M)$
such that $\Lap f \in H^1(M)$,
\[ \int_M \bigg\{ D[\Lap f](\Grad f) +K F^2(\Grad f)
 +\frac{(\Lap f)^2}{N} \bigg\} \,d\fm \le 0. \]
In particular, if $K>0$, then we have
\begin{equation}\label{eq:Lich-}
\int_M F^2(\Grad f) \,d\fm \le \frac{N-1}{KN} \int_M (\Lap f)^2 \,d\fm.
\end{equation}
If $N=\infty$, then the coefficient in the RHS of \eqref{eq:Lich-} is read as $1/K$.
\end{proposition}

\begin{proof}
The first assertion is a direct consequence of Theorem~\ref{th:BWint} with $\phi \equiv 1$.
We further observe, by the integration by parts,
\[ K \int_M F^2(\Grad f) \,d\fm +\frac{\| \Lap f \|_{L^2}^2}{N}
 \le -\int_M D[\Lap f](\Grad f) \,d\fm
 =\| \Lap f \|_{L^2}^2. \]
Rearranging this inequality yields \eqref{eq:Lich-} when $K>0$.
$\qedd$
\end{proof}

Now the Poincar\'e--Lichnerowicz inequality is obtained
by a technique somewhat related to the proof of Theorem~\ref{th:L2}.
We define the \emph{variance} of $f \in L^2(M)$ (under $\fm(M)=1$) as
\[ \Var_{\fm}(f) :=\int_M f^2 \,d\fm -\bigg( \int_M f \,d\fm \bigg)^2. \]

\begin{theorem}[Poincar\'e--Lichnerowicz inequality]\label{th:Lich}
Let $M$ be compact, $\fm(M)=1$,
and suppose $\Ric_N \ge K>0$ for some $N \in (-\infty,0) \cup [n,\infty]$.
Then we have, for any $f \in H^1(M)$,
\[ \Var_{\fm}(f) \le \frac{N-1}{KN} \int_M F^2(\Grad f) \,d\fm. \]
\end{theorem}

\begin{proof}
Let $(u_t)_{t \ge 0}$ be the solution to the heat equation with $u_0=f$,
and put $\Phi(t):=\|u_t\|_{L^2}^2$ for $t \ge 0$.
Observe first that the \emph{ergodicity}
\begin{equation}\label{eq:ergo}
u_t\ \to\ \int_M u_0 \,d\fm \quad \text{in}\ L^2(M)
\end{equation}
holds since $\lim_{t \to \infty}\cE(u_t)=0$.
It follows from the heat equation that
\[ \Phi'(t) =2\int_M u_t \Lap u_t \,d\fm
 =-2\int_M F^2(\Grad u_t) \,d\fm =-4\cE(u_t), \]
as well as
\[ \lim_{\delta \downarrow 0}\frac{\Phi'(t+\delta)-\Phi'(t)}{\delta}
 =-4 \lim_{\delta \downarrow 0}\frac{\cE(u_{t+\delta})-\cE(u_t)}{\delta}
 =4\| \Lap u_t \|_{L^2}^2 \]
for all $t>0$.
Then \eqref{eq:Lich-} implies
\begin{equation}\label{eq:Lich0}
-2\Phi'(t) \le \frac{N-1}{KN} \lim_{\delta \downarrow 0}\frac{\Phi'(t+\delta)-\Phi'(t)}{\delta}.
\end{equation}
Notice also that the ergodicity \eqref{eq:ergo} implies
$\lim_{t \to \infty} \Phi(t) =(\int_M f \,d\fm)^2$.
Thus the differential inequality \eqref{eq:Lich0} yields
\[ \Var_{\fm}(f) =-\int_0^{\infty} \Phi'(t) \,dt
 \le \frac{N-1}{2KN} \bigg( \lim_{t \to \infty}  \Phi'(t)-\Phi'(0) \bigg)
 =\frac{2(N-1)}{KN} \cE(f). \]
This completes the proof.
$\qedd$
\end{proof}

\subsection{Logarithmic Sobolev inequality}\label{ssc:log}

We next study the \emph{logarithmic Sobolev inequality}.
From here on we consider only $N \in [n,\infty)$ with $K>0$
(actually the inequality fails for $N<0$, see Remark~\ref{rm:LSI} below).
See \cite{Oint} for the case of $N=\infty$.
Recall from Theorem~\ref{th:wBM} that $\Ric_N \ge K$ implies the compactness of $M$,
thus we normalize $\fm$ as $\fm(M)=1$ without loss of generality.
We first consider a sufficient condition for the logarithmic Sobolev inequality.

\begin{proposition}\label{pr:5.7.3}
Assume that $M$ is compact, $\Ric_{\infty} \ge K>0$, and
\begin{equation}\label{eq:5.7.3}
\int_M \frac{F^2(\Grad u)}{u} \,d\fm
 \le -C \int_M \bigg\{ Du \bigg( \nabla^{\Grad u} \bigg[ \frac{F^2(\Grad u)}{2u^2} \bigg] \bigg)
 +u D[\Lap(\log u)](\Grad[\log u]) \bigg\} \,d\fm
\end{equation}
holds for some constant $C>0$ and all functions $u \in H^2(M) \cap \cC^1(M)$
such that $\Lap u \in H^1(M)$ and $\inf_M u>0$.
Then the logarithmic Sobolev inequality
\[ \int_{\{f>0\}} f \log f \,d\fm \le \frac{C}{2} \int_{\{f>0\}} \frac{F^2(\Grad f)}{f} \,d\fm \]
holds for all nonnegative functions $f \in H^1(M)$ with $\int_M f \,d\fm=1$.
\end{proposition}

\begin{outline}
We shall assume $\inf_M f>0$ and
use a similar method to the proof of Theorem~\ref{th:Lich} for
the \emph{relative entropy}
\[ \Psi(t) =\Ent_{\fm}(u_t \fm) :=\int_M u_t \log u_t \,d\fm, \]
where $(u_t)_{t \ge 0}$ is the solution to the heat equation with $u_0=f$.
The inequality \eqref{eq:5.7.3} then yields the differential inequality $-2\Psi'(t) \le C\Psi''(t)$.
Together with $\lim_{t \to \infty} \Psi(t)=\lim_{t \to \infty} \Psi'(t)=0$,
we obtain
\[ \int_M f \log f \,d\fm =-\int_0^{\infty} \Psi'(t) \,dt
 \le \frac{C}{2} \int_0^{\infty} \Psi''(t) \,dt =-\frac{C}{2} \Psi'(0). \]
This is indeed the desired inequality.
$\qedd$
\end{outline}

See \cite{Ofunc} and \cite{BGL} for the omitted calculations
in the proofs of the proposition above and the theorem below.

\begin{theorem}[Logarithmic Sobolev inequality]\label{th:LSI}
Assume that $\Ric_N \ge K>0$ for some $N \in [n,\infty)$ and $\fm(M)=1$.
Then we have
\[ \int_{\{f>0\}} f \log f \,d\fm \le \frac{N-1}{2KN} \int_{\{f>0\}} \frac{F^2(\Grad f)}{f} \,d\fm \]
for all nonnegative functions $f \in H^1(M)$ with $\int_M f \,d\fm=1$.
\end{theorem}

\begin{outline}
Fix $h \in \cC^{\infty}(M)$ and consider the function $\e^{ah}$ for $a>0$ (chosen later).
For brevity let us introduce the notation common in the $\Gamma$-calculus:
\[ \Gamma_2(h) :=\Delta\!^{\Grad h} \bigg[ \frac{F^2(\Grad h)}{2} \bigg]
 -D[\Lap h](\Grad h). \]
On the one hand, by the chain rule and $a>0$, we calculate
\[ \Gamma_2(\e^{ah}) =a^2 \e^{2ah}
 \big\{ \Gamma_2(h) +aD[F^2(\Grad h)](\Grad h) +a^2 F^4(\Grad h) \big\} \]
(notice that $a>0$ yields $\Grad(\e^{ah})=a\e^{ah} \Grad h$).
On the other hand, it follows from the integration by parts that
\[ \int_M \Gamma_2(\e^{ah}) \,d\fm =a^2 \int_M \e^{2ah}
 \big\{ (\Lap h)^2 -2aD[F^2(\Grad h)](\Grad h) -3a^2 F^4(\Grad h) \big\} \,d\fm. \]
Comparing these yields
\begin{equation}\label{eq:ah}
\int_M \e^{2ah} (\Lap h)^2 \,d\fm =\int_M \e^{2ah}
 \big\{ \Gamma_2(h) +3aD[F^2(\Grad h)](\Grad h) +4a^2 F^4(\Grad h) \big\} \,d\fm.
\end{equation}

We apply the Bochner inequality \eqref{eq:Boc} to $\e^{ah}$,
now written as $\Gamma_2(\e^{ah}) \ge KF^2(\Grad \e^{ah}) +(\Lap \e^{ah})^2/N$,
and see
\begin{align}
&\Gamma_2(h) +aD[F^2(\Grad h)](\Grad h) +a^2 F^4(\Grad h) \nonumber\\
&\ge KF^2(\Grad h) +\frac{1}{N}
 \big\{ (\Lap h)^2 +2aF^2(\Grad h) \Lap h +a^2 F^4(\Grad h) \big\} \label{eq:4.7}
\end{align}
in the weak sense.
Integrating this inequality multiplied by the test function $\e^h$,
using \eqref{eq:ah} with $a=1/2$ and rearranging, we obtain
\begin{align}
\bigg( 1-\frac{1}{N} \bigg) \int_M \e^h \Gamma_2(h) \,d\fm
&\ge K\int_M \e^h F^2(\Grad h) \,d\fm \nonumber\\
&\quad +\frac{3-2(N+2)a}{2N} \int_M \e^h D[F^2(\Grad h)](\Grad h) \,d\fm \nonumber\\
&\quad +\frac{(a-1)^2-Na^2}{N} \int_M \e^h F^4(\Grad h) \,d\fm. \label{eq:4.8}
\end{align}
We finally choose $a=3/\{2(N+2)\}>0$ and conclude that
\[ \frac{N-1}{N} \int_M \e^h \Gamma_2(h) \,d\fm
 \ge K\int_M \e^h F^2(\Grad h) \,d\fm. \]
This is the desired inequality \eqref{eq:5.7.3} for $u=\e^h$.
We complete the proof by an approximation argument.
$\qedd$
\end{outline}

By the standard implication going back to \cite{OV},
we have the \emph{Talagrand inequality} as a corollary.

\begin{corollary}[Talagrand inequality]\label{cr:Tala}
Assume that $\Ric_N \ge K>0$ for $N \in [n,\infty)$ and $\fm(M)=1$.
Then we have, for all $\mu \in \cP(M)$,
\[ W_2^2(\mu,\fm) \le \frac{2(N-1)}{KN} \Ent_{\fm}(\mu), \]
where $W_2$ is the $L^2$-Wasserstein distance.
\end{corollary}

\begin{remark}\label{rm:LSI}
Different from the Poincar\'e--Lichnerowicz inequality in the previous section,
$N$ cannot be negative in the logarithmic Sobolev and Talagrand inequalities.
This is because the Talagrand inequality implies the \emph{normal concentration} of $\fm$ (\cite{Led}),
while the model space in \cite{Mineg,Oneg} enjoys only the \emph{exponential concentration}.
See \cite{Ma} for further details.
\end{remark}

\subsection{Sobolev inequalities}\label{ssc:Sobo}

Our next object is the \emph{Sobolev inequality}.
We will first obtain a non-sharp inequality
followed by some qualitative consequences.
Then, with the help of those properties, we proceed to the sharp estimate.
The resulting inequality (Theorem~\ref{th:Sobo}) is the same as the Riemannian case,
however, we will need an additional care on the admissible range of the exponent $p$
when $F$ is non-reversible.

\begin{proposition}[Logarithmic entropy-energy inequality]\label{pr:6.8.1}
Assume $\Ric_N \ge K>0$ for $N \in [n,\infty)$ and $\fm(M)=1$.
Then we have
\[ \Ent_{\fm}(f^2 \fm)
 \le \frac{N}{2}\log \bigg( 1+\frac{4}{KN} \int_M F^2(\Grad f) \,d\fm \bigg) \]
for all $f \in H^1(M)$ with $\int_M f^2 \,d\fm=1$.
\end{proposition}

\begin{outline}
It suffices to consider the case where $c \le f \le C$ for some $0<c<C<\infty$.
Let $(u_t)_{t \ge 0}$ be the solution to the heat equation with $u_0=f^2$,
and put $\Psi(t):=\Ent_{\fm}(u_t \fm)$ similarly to Proposition~\ref{pr:5.7.3}.
By the Bochner inequality \eqref{eq:Boc} and the Cauchy--Schwarz inequality
we find
\[ \Psi''(t) \ge -2K\Psi'(t) +\frac{2}{N} \Psi'(t)^2. \]
This implies that the function
\[ t \ \longmapsto \ \e^{-2Kt} \bigg( \frac{1}{N}-\frac{K}{\Psi'(t)} \bigg) \]
is non-decreasing in $t>0$, and hence
\[ -\Psi'(t) \le KN \bigg\{ \e^{2Kt} \bigg( 1-\frac{KN}{\Psi'(0)} \bigg) -1 \bigg\}^{-1}. \]
Integrating this inequality gives
\[ \Psi(0) -\Psi(t) \le
 \frac{N}{2} \log\bigg( 1-(1-\e^{-2Kt}) \frac{\Psi'(0)}{KN} \bigg). \]
Since
\[ \Psi'(0) =-\int_M \frac{F^2(\Grad (f^2))}{f^2} \,d\fm
 =-4\int_M F^2(\Grad f) \,d\fm, \]
letting $t \to \infty$ completes the proof.
$\qedd$
\end{outline}

The above inequality yields the \emph{Nash inequality}
and then a non-sharp Sobolev inequality.

\begin{lemma}[Nash inequality]\label{lm:Nash}
Assume that $\Ric_N \ge K>0$ for $N \in [n,\infty)$ and $\fm(M)=1$.
Then we have, for all $f \in H^1(M)$,
\[ \|f\|_{L^2}^{N+2} \le
 \bigg( \|f\|_{L^2}^2 +\frac{4}{KN} \cE(f) \bigg)^{N/2} \|f\|_{L^1}^2. \]
\end{lemma}

\begin{outline}
Normalize $f$ so as to satisfy $\int_M f^2 \,d\fm=1$,
and put $\psi(\theta):=\log(\|f\|_{L^{1/\theta}})$ for $\theta \in (0,1]$.
We see by the H\"older inequality that $\psi$ is a convex function.
Therefore
\[ \psi(1) \ge \psi\bigg( \frac{1}{2} \bigg) +\frac{1}{2} \psi'\bigg( \frac{1}{2} \bigg)
 =\frac{1}{2} \psi'\bigg( \frac{1}{2} \bigg)
 =-\frac{1}{2} \Ent_{\fm}(f^2 \fm). \]
Combining this with Proposition~\ref{pr:6.8.1} gives the claim.
$\qedd$
\end{outline}

\begin{proposition}[Non-sharp Sobolev inequality]\label{pr:6.2.3}
Assume that $\Ric_N \ge K>0$ for $N \in [n,\infty) \cap (2,\infty)$ and $\fm(M)=1$.
Then we have
\[ \|f\|_{L^p}^2 \le C_1 \|f\|_{L^2}^2 +C_2 \cE(f) \]
for all $f \in H^1(M)$,
where $p=2N/(N-2)$, $C_1=C_1(N)>1$ and $C_2=C_2(K,N)>0$.
\end{proposition}

\begin{outline}
We can assume $c \le f \le C$ for some $0<c<C<\infty$.
Slice $f$ into
\[ f_k(x) :=\left\{ \begin{array}{cl}
 2^k & \text{if}\ f(x) >2^{k+1}, \smallskip\\
 f-2^k & \text{if}\ 2^k< f(x) \le 2^{k+1}, \smallskip\\
 0 & \text{if}\ f(x) \le 2^k,
 \end{array} \right.  \]
for $k \in \Z$.
We show the claim by applying the Nash inequality to each $f_k$.
$\qedd$
\end{outline}

The above Sobolev inequality has some applications
those will be used to show the sharp inequality.
Actually one can reduce such qualitative arguments to the Riemannian case
by observing the following.

\begin{corollary}\label{cr:nonSob}
Assume that $\Ric_N \ge K>0$ for $N \in [n,\infty) \cap (2,\infty)$ and $\fm(M)=1$.
Then there exists a $\cC^{\infty}$-Riemannian metric $g$ for which
\[ \|f\|_{L^p}^2 \le C_1 \|f\|_{L^2}^2 +C_2 \sS_F \cE^g(f) \]
holds for all $f \in H^1(M)$,
where $\cE^g$ is the energy form of $(M,g,\fm)$
and $p=2N/(N-2)$, $C_1>1$ and $C_2>0$ are as in Proposition~$\ref{pr:6.2.3}$.
\end{corollary}

We are now ready to show the sharp Sobolev inequality.

\begin{theorem}[Sobolev inequality]\label{th:Sobo}
Assume that $\Ric_N \ge K>0$ for $N \in [n,\infty)$ and $\fm(M)=1$.
Then we have
\[ \frac{\|f\|_{L^p}^2 -\|f\|_{L^2}^2}{p-2} \le \frac{N-1}{KN} \int_M F^2(\Grad f) \,d\fm \]
for all $1 \le p \le 2(N+1)/N$ and $f \in H^1(M)$.
\end{theorem}

The case of $p=2$ is understood as the limit, giving the logarithmic Sobolev inequality
(Theorem~\ref{th:LSI}).
The $p=1$ case amounts to the Poincar\'e--Lichnerowicz inequality (Theorem~\ref{th:Lich}).

\begin{outline}
Take the smallest possible constant $C>0$ satisfying
\begin{equation}\label{eq:C}
\frac{\|f\|_{L^p}^2 -\|f\|_{L^2}^2}{p-2} \le 2C\cE(f)
\end{equation}
for all nonnegative functions $f \in H^1(M)$.
In order to show $C \le (N-1)/KN$,
let us suppose that there is an extremal (nonconstant) function $f \ge 0$
enjoying equality in \eqref{eq:C} as well as $0<c \le f \le C<\infty$.
We normalize $f$ as $\|f\|_{L^p}=1$.
The equality in \eqref{eq:C} implies
\[ f^{p-1} -f=-C(p-2) \Lap f, \]
which improves the regularity of $f$.
Put $u:=\log f$.

On the one hand, for $b \ge 0$,
it follows from \eqref{eq:ah} (with $a=b/2$) and the integration by parts that
\begin{align}
C\int_M \e^{bu} \Gamma_2(u) \,d\fm
&= \int_M \e^{bu} F^2(\Grad u) \,d\fm \nonumber\\
&\quad +C\bigg( p-1-\frac{b}{2} \bigg) \int_M \e^{bu} D[F^2(\Grad u)](\Grad u) \,d\fm \nonumber\\
&\quad +C(p-2) (b-1) \int_M \e^{bu} F^4(\Grad u) \,d\fm. \label{eq:6.8.5}
\end{align}
We remark that $b \ge 0$ was required to apply \eqref{eq:ah}.
On the other hand, we deduce from \eqref{eq:4.7} that
\begin{align}
\bigg( 1-\frac{1}{N} \bigg) \int_M \e^{bu} \Gamma_2(u) \,d\fm
&\ge K \int_M \e^{bu} F^2(\Grad u) \,d\fm \nonumber\\
&\quad +\bigg( \frac{3b-4a}{2N}-a \bigg) \int_M \e^{bu} D[F^2(\Grad u)](\Grad u) \,d\fm \nonumber\\
&\quad +\bigg( \frac{(a-b)^2}{N}-a^2 \bigg) \int_M \e^{bu} F^4(\Grad u) \,d\fm,
 \label{eq:6.8.6}
\end{align}
which can be seen as a variant of \eqref{eq:4.8}.
Comparing \eqref{eq:6.8.5} and \eqref{eq:6.8.6},
we would like to choose $a$ and $b$ enjoying
\[ p-1-\frac{b}{2} =\frac{3b-2(N+2)a}{2(N-1)}, \qquad
 (p-2)(b-1) =\frac{(a-b)^2 -Na^2}{N-1} \]
as well as $a \ge 0$, $b \ge 0$.
This is possible for $p \in [1,2(N+1)/N]$,
and then we conclude $C \le (N-1)/KN$ as desired.

There remains the delicate issue that
such an extremal function $f$ does not necessarily exist.
Hence one needs an extra discussion on the approximation.
This technical argument can be reduced to the Riemannian case
thanks to the non-sharp Sobolev inequality in Corollary~\ref{cr:nonSob}.
See \cite[Theorem~6.8.3]{BGL} for details.
$\qedd$
\end{outline}

\begin{remark}\label{rm:Sobo}
The admissible range of the exponent $p$ is $[1,2N/(N-2)]$
in the Riemannian and the reversible Finsler cases.
In the above non-reversible situation,
the additional constraints $a \ge 0$ and $b \ge 0$ make the range narrower.
One can slightly extend the range $[1,2(N+1)/N]$ in Theorem~\ref{th:Sobo}
by a careful calculation into
\[ \bigg[ 1,\frac{7N^2 +2N +(N+2)\sqrt{N^2 +8N}}{4N(N-1)} \bigg], \]
though it is still smaller than $[1,2N/(N-2)]$.
\end{remark}

\section{Gaussian isoperimetric inequality}\label{sc:isop}

This final section is devoted to a geometric application of the $\Gamma$-calculus,
the \emph{Gaussian isoperimetric inequality} obtained in \cite{Oisop}.
This is the infinite dimensional counterpart to the \emph{L\'evy--Gromov isoperimetric inequality},
and was first established in the Riemannian setting by Bakry--Ledoux \cite{BL}.
We will give historical background in \S \ref{ssc:back} and the outline of the proof in \S \ref{ssc:isop}.

\subsection{Background for L\'evy--Gromov isoperimetic inequality}\label{ssc:back}

Given a Finsler manifold $(M,F,\fm)$ such that $\fm(M)=1$,
we define the \emph{isoperimetric profile}
$\cI_{(M,F,\fm)}:[0,1] \lra [0,\infty]$ by
\[ \cI_{(M,F,\fm)}(\theta)
 :=\inf\{ \fm^+(A) \,|\, A \subset M:\text{Borel set with}\ \fm(A)=\theta \}, \]
where
\[ \fm^+(A):=\liminf_{\ve \downarrow 0} \frac{\fm(B^+(A,\ve))-\fm(A)}{\ve},
 \quad B^+(A,\ve):=\{ y \in M \,|\, \inf_{x \in A}d(x,y)<\ve \}. \]
In the Riemannian case,
the classical theorem of L\'evy and Gromov \cite{Lev1,Lev2,Gr} asserts that,
for an $n$-dimensional Riemannian manifold $(M,g)$ of $\Ric \ge n-1$
with the normalized volume measure $\fm:=\vol_g(M)^{-1} \,\vol_g$,
the isoperimetric profile $\cI_{(M,g,\fm)}$ is bounded from below
by the profile of the unit sphere $\Sph^n$:
\begin{equation}\label{eq:LG}
\cI_{(M,g,\fm_g)}(\theta) \ge \cI_{(\Sph^n,\fm_{\Sph^n})}(\theta),
\end{equation}
where $\fm_{\Sph^n}$ is the normalized volume measure as well.
The standard strategy of the proof of \eqref{eq:LG} is as follows:
\begin{enumerate}[(1)]
\item
We take an extremal region $A \subset M$
achieving the minimal boundary measure $\fm^+(A)$ with the prescribed volume $\theta$.

\item
Then the deep theorem in \emph{geometric measure theory} (\`a la Federer, Almgren et al)
guarantees a certain regularity of the boundary of $A$.

\item
Perturbing $A$ gives a differential inequality
(a \emph{Heintze--Karcher type inequality}).

\item
Combining the above differential inequality with the analysis of the behavior
as $\theta \downarrow 0$ gives the isoperimetic inequality.
\end{enumerate}

Along the same strategy one can study the weighted version (\cite{Bay}) and,
moreover, the combinations of upper diameter bounds and lower curvature bounds (\cite{Misharp}).
The most general work of Milman \cite{Misharp} gave the sharp estimate:
\begin{equation}\label{eq:CDD}
\cI_{(M,g,\fm)}(\theta) \ge \cI_{K,N,D}(\theta)
\end{equation}
for $(M,g,\fm)$ with $\Ric_N \ge K$ and $\diam M \le D$,
where $\cI_{K,N,D}$ is the explicit function.
Up to now, however, the regularity theory is known only for Riemannian manifolds.
This had been an obstacle for generalizations to Finsler manifolds as well as
less smooth spaces such as metric measure spaces.

In 2014, Klartag \cite{Kl} gave a beautiful alternative proof
of the L\'evy--Gromov isoperimetric inequality, still on weighted Riemannian manifolds,
but without the regularity theory.
His breakthrough was done by generalizing the \emph{localization method}
in convex geometry to Riemannian manifolds with the help of optimal transport theory.
The localization method, going back to \cite{PW,GM,LS,KLS},
is a sophisticated tool reducing an inequality to those on geodesics.
Then the analysis becomes much simpler and clearer.

Inspired by \cite{Kl}, Cavalletti--Mondino generalized the localization method
to essentially non-branching metric measure spaces satisfying the curvature-dimension condition,
and showed the isoperimetric inequality \eqref{eq:CDD} in \cite{CMisop}
and several functional inequalities in \cite{CMfunc}.
This class of spaces includes reversible Finsler manifolds.
In \cite{Oneedle}, we generalized the argument in \cite{CMisop}
to possibly non-reversible Finsler manifolds, however,
then it turned out that the localization method gives only a non-sharp estimate
in the non-reversible case, precisely,
\begin{equation}\label{eq:needle}
\cI_{(M,F,\fm)}(\theta) \ge \Lambda_F^{-1} \cdot \cI_{K,N,D}(\theta).
\end{equation}
This is due to the fact that reverse curves of geodesics are not necessarily geodesic.
It seems plausible to expect that 
$\Lambda_F^{-1}$ in \eqref{eq:needle} would be removed,
that is to say, non-reversible Finsler manifolds enjoy
the same isoperimetric inequality as reversible Finsler manifolds.

Towards this direction,
we consider the special case where $N=D=\infty$ and $K>0$.
This is the only case where we have another alternative proof of \eqref{eq:LG},
based on the $\Gamma$-calculus (\cite{BL}).
As we saw in the previous section, the $\Gamma$-calculus is not sensitive
to the non-reversibility (only the exception was the range of $p$ in Theorem~\ref{th:Sobo}),
and we actually obtain the sharp estimate.

\subsection{Isoperimetric inequality}\label{ssc:isop}

In order to state the key estimate which is a kind of gradient estimate,
we define
\[ \varphi(c):=\frac{1}{\sqrt{2\pi}} \int_{-\infty}^c \e^{-b^2/2} \,db
 \quad \text{for}\ c \in \R, \qquad
 \scN(\theta):=\varphi' \circ \varphi^{-1}(\theta) \quad \text{for}\ \theta \in (0,1). \]
Set also $\scN(0)=\scN(1):=0$.

\begin{theorem}\label{th:key}
Assume $\Ric_{\infty} \ge K$ for some $K \in \R$ and $\sS_F<\infty$.
Then we have, given a global solution $(u_t)_{t \ge 0}$ to the heat equation
with $u_0 \in \cC^{\infty}_c(M)$ and $0 \le u_0 \le 1$,
\begin{equation}\label{eq:key}
\sqrt{\scN^2(u_t) +\alpha F^2(\Grad u_t)}
 \le P^{\Grad u}_{0,t} \Big( \sqrt{\scN^2(u_0) +c_{\alpha}(t) F^2(\Grad u_0)} \Big)
\end{equation}
for all $\alpha \ge 0$ and $t>0$, where
\[ c_{\alpha}(t):=\frac{1-\e^{-2Kt}}{K} +\alpha \e^{-2Kt} >0 \]
and $c_{\alpha}(t):=2t+\alpha$ when $K=0$.
\end{theorem}

For simplicity, we suppressed the dependence of $c_{\alpha}$ on $K$.

\begin{outline}
By the construction of global solutions as gradient curves,
we find that $0 \le u_0 \le 1$ implies $0 \le u_t \le 1$ for all $t>0$,
and hence $\scN(u_t)$ makes sense.
Fix $t>0$ and put
\[ \zeta_s:=\sqrt{\scN^2(u_s) +c_{\alpha}(t-s) F^2(\Grad u_s)},
 \qquad 0 \le s \le t. \]
Then \eqref{eq:key} is written as $\zeta_t \le P^{\Grad u}_{0,t}(\zeta_0)$,
thereby we are done when we show $\del_s[P^{\Grad u}_{s,t}(\zeta_s)] \le 0$.
We can see by a simple calculation that
\[ \del_s[P^{\Grad u}_{s,t}(\zeta_s)]
 =P^{\Grad u}_{s,t}(\del_s \zeta_s -\Delta\!^{\Grad u_s} \zeta_s). \]
Hence it is sufficient to prove $\Delta\!^{\Grad u_s} \zeta_s -\del_s \zeta_s \ge 0$
for $0<s<t$.

By using $c'_{\alpha}(t)=2(1-Kc_{\alpha}(t))$, $\scN''=-1/\scN$
and the improved Bochner inequality (Proposition~\ref{pr:iBoch}),
we obtain
\begin{align*}
\Delta\!^{\Grad u_s} \zeta_s -\del_s \zeta_s
&\ge \frac{c_{\alpha}(t-s) \scN'(u_s)^2}{\zeta_s^3} F^4(\Grad u_s) \\
&\quad -\frac{c_{\alpha}(t-s) \scN(u_s) \scN'(u_s)}{\zeta_s^3}
 Du_s \big( \nabla^{\Grad u_s}[F^2(\Grad u_s)] \big) \\
&\quad +\frac{c_{\alpha}(t-s)}{\zeta_s^3} \frac{\scN^2(u_s)}{F^2(\Grad u_s)}
 D\bigg[ \frac{F^2(\Grad u_s)}{2} \bigg]
 \bigg( \nabla^{\Grad u_s}\bigg[ \frac{F^2(\Grad u_s)}{2} \bigg] \bigg).
\end{align*}
Combining this with the Cauchy--Schwarz inequality
\[ Du_s \big( \nabla^{\Grad u_s}[F^2(\Grad u_s)] \big)
 \le F(\Grad u_s) \sqrt{D[F^2(\Grad u_s)] \big( \nabla^{\Grad u_s} [F^2(\Grad u_s)] \big)} \]
for $g_{\Grad u_s}$, we conclude that
\begin{align*}
&\Delta\!^{\Grad u_s} \zeta_s -\del_s \zeta_s \\
&\ge \frac{c_{\alpha}(t-s)}{\zeta_s^3}
 \bigg( |\scN'(u_s)| F^2(\Grad u_s) -\frac{\scN(u_s)}{2F(\Grad u_s)}
 \sqrt{D[F^2(\Grad u_s)] \big( \nabla^{\Grad u_s} [F^2(\Grad u_s)] \big)} \bigg)^2 \\
& \ge 0.
\end{align*}
This completes the proof.
$\qedd$
\end{outline}

If $K>0$, then $\Ric_{\infty} \ge K$ implies $\fm(M)<\infty$
and hence we can normalize $\fm$ (see \cite{StI}).
Choosing $\alpha=K^{-1}$ in \eqref{eq:key} and letting $t \to \infty$
yields the following.

\begin{corollary}\label{cr:key}
Assume that $(M,F)$ is complete and satisfies
$\Ric_{\infty} \ge K>0$, $\sS_F<\infty$ and $\fm(M)=1$.
Then we have, for any $u \in \cC_c^{\infty}(M)$ with $0 \le u \le 1$,
\begin{equation}\label{eq:key'}
\sqrt{K} \scN\bigg( \int_M u \,d\fm \bigg)
 \le \int_M \sqrt{K\scN^2(u) +F^2(\Grad u)} \,d\fm.
\end{equation}
\end{corollary}

Now we are ready to show the isoperimetric inequality.

\begin{theorem}[Gaussian isoperimetric inequality]\label{th:isop}
Let $(M,F,\fm)$ be complete and satisfy $\Ric_{\infty} \ge K>0$, $\fm(M)=1$
and $\sS_F<\infty$.
Then we have
\begin{equation}\label{eq:BL}
\cI_{(M,F,\fm)}(\theta) \ge \cI_K(\theta)
\end{equation}
for all $\theta \in [0,1]$, where
\[ \cI_K(\theta):=\sqrt{\frac{K}{2\pi}} \e^{-Kc^2(\theta)/2} \qquad
 \text{with}\ \ \theta=\int_{-\infty}^{c(\theta)} \sqrt{\frac{K}{2\pi}} \e^{-Ka^2/2} \,da. \]
\end{theorem}

\begin{outline}
Let $\theta \in (0,1)$.
Fix a closed set $A \subset M$ with $\fm(A)=\theta$ and consider
\[ u^{\ve}(x):=\max\{1-\ve^{-1}d(x,A),0\}, \qquad \ve>0. \]
Notice that $F(\Grad u^{\ve})=\ve^{-1}$ on $B^-(A,\ve) \setminus A$, where
\[ B^-(A,\ve):=\Big\{ x \in M \,\Big|\, \inf_{y \in A} d(x,y)<\ve \Big\}. \]
Applying \eqref{eq:key'} to (smooth approximations of) $u^{\ve}$ and letting $\ve \downarrow 0$ implies,
with the help of $\scN(0)=\scN(1)=0$,
\[ \sqrt{K} \scN(\theta)
 \le \liminf_{\ve \downarrow 0} \frac{\fm(B^-(A,\ve))-\fm(A)}{\ve}. \]
This is the desired isoperimetric inequality for the reverse Finsler structure $\rev{F}$
(recall Definition~\ref{df:rev}).
Because the curvature bound $\Ric_{\infty} \ge K$ is common to $F$ and $\rev{F}$,
we also obtain \eqref{eq:BL}.
$\qedd$
\end{outline}

The inequality \eqref{eq:BL} has the same form as the Riemannian case in \cite{BL},
thus it is sharp and a model space is the real line $\R$
equipped with the normal (Gaussian) distribution
$d\fm=\sqrt{K/2\pi} \, \e^{-Kx^2/2} \,dx$.
See \cite{BL} for the original work on general linear diffusion semigroups
(influenced by Bobkov's works \cite{Bo1,Bo2}),
\cite{Bor,SC} for the classical Euclidean or Hilbert cases,
and \cite{AM} for the recent result on $\RCD(K,\infty)$-spaces.


\begin{thebibliography}{Oh99}

\bibitem[AT]{AT}
J.~C.~\'Alvarez-Paiva and A.~C.~Thompson,
Volumes in normed and Finsler spaces.
A sampler of Riemann--Finsler geometry, 1--48,
Math.\ Sci.\ Res.\ Inst.\ Publ., {\bf 50}, Cambridge Univ.\ Press, Cambridge, 2004.

\bibitem[AGS]{AGSrcd}
L.~Ambrosio, N.~Gigli and G.~Savar\'e,
Metric measure spaces with Riemannian Ricci curvature bounded from below.
Duke Math.\ J.\ {\bf 163} (2014), 1405--1490.

\bibitem[AM]{AM}
L.~Ambrosio and A.~Mondino,
Gaussian-type isoperimetric inequalities in $\RCD(K,\infty)$ probability spaces for positive $K$.
Atti Accad.\ Naz.\ Lincei Rend.\ Lincei Mat.\ Appl.\ {\bf 27} (2016), 497--514.

\bibitem[Au]{Au}
L.~Auslander, On curvature in Finsler geometry.
Trans.\ Amer.\ Math.\ Soc.\ {\bf 79} (1955), 378--388.

\bibitem[Bak]{Ba}
D.~Bakry,
L'hypercontractivit\'e et son utilisation en th\'eorie des semigroupes. (French)
Lectures on probability theory (Saint-Flour, 1992), 1--114,
Lecture Notes in Math., {\bf 1581}, Springer, Berlin, 1994.

\bibitem[BE]{BE}
D.~Bakry and M.~\'Emery, Diffusions hypercontractives. (French)
S\'eminaire de probabilit\'es, XIX, 1983/84, 177--206,
Lecture Notes in Math., {\bf 1123}, Springer, Berlin, 1985.

\bibitem[BGL]{BGL}
D.~Bakry, I.~Gentil and M.~Ledoux,
Analysis and geometry of Markov diffusion operators.
Springer, Cham, 2014.

\bibitem[BL]{BL}
D.~Bakry and M.~Ledoux,
L\'evy--Gromov's isoperimetric inequality for an infinite-dimensional diffusion generator.
Invent.\ Math.\ {\bf 123} (1996), 259--281.

\bibitem[BQ]{BQ}
D.~Bakry and Z.~Qian,
Some new results on eigenvectors via dimension, diameter, and Ricci curvature.
Adv.\ Math.\ {\bf 155} (2000), 98--153.

\bibitem[BCL]{BCL}
K.~Ball, E.~A.~Carlen and E.~H.~Lieb,
Sharp uniform convexity and smoothness inequalities for trace norms.
Invent.\ Math.\ {\bf 115} (1994), 463--482.

\bibitem[BCS]{BCS}
D.~Bao, S.-S.~Chern and Z.~Shen, An introduction to Riemann-Finsler geometry.
Springer-Verlag, New York, 2000.

\bibitem[Bay]{Bay}
V.~Bayle, Propri\'et\'es de concavit\'e du profil isop\'erim\'etrique et applications (French).
Th\'ese de Doctorat, Institut Fourier, Universite Joseph-Fourier, Grenoble, 2003.

\bibitem[Bob1]{Bo1}
S.~Bobkov,
A functional form of the isoperimetric inequality for the Gaussian measure.
J.\ Funct.\ Anal.\ {\bf 135} (1996), 39--49.

\bibitem[Bob2]{Bo2}
S.~G.~Bobkov,
An isoperimetric inequality on the discrete cube, and an elementary proof of the isoperimetric inequality in Gauss space.
Ann.\ Probab.\ {\bf 25} (1997), 206--214.

\bibitem[Bor]{Bor}
C.~Borell, The Brunn--Minkowski inequality in Gauss space.
Invent.\ Math.\ {\bf 30} (1975), 207--216.

\bibitem[CM1]{CMisop}
F.~Cavalletti and A.~Mondino,
Sharp and rigid isoperimetric inequalities in metric-measure spaces with lower Ricci curvature bounds.
Invent.\ Math.\ (to appear). Available at {\sf arXiv:1502.06465}

\bibitem[CM2]{CMfunc}
F.~Cavalletti and A.~Mondino,
Sharp geometric and functional inequalities in metric measure spaces with lower Ricci curvature bounds.
Geom.\ Topol.\ {\bf 21} (2017), 603--645.

\bibitem[Ch${}_+$]{Ch+}
B.~Chow, S.-C.~Chu, D.~Glickenstein, C.~Guenther, J.~Isenberg, T.~Ivey, D.~Knopf, P.~Lu, F.~Luo, L.~Ni,
The Ricci flow: techniques and applications.~Part I.~Geometric aspects.
American Mathematical Society, Providence, RI, 2007.

\bibitem[EKS]{EKS}
M~Erbar, K.~Kuwada and K.-T.~Sturm,
On the equivalence of the entropic curvature-dimension condition
and Bochner's inequality on metric measure spaces.
Invent.\ Math.\ {\bf 201} (2015), 993--1071.

\bibitem[Ev]{Ev}
L.~C.~Evans, Partial differential equations.
American Mathematical Society, Providence, RI, 1998.

\bibitem[FLZ]{FLZ}
F.~Fang, X.-D.~Li and Z.~Zhang,
Two generalizations of Cheeger-Gromoll splitting theorem via Bakry--Emery Ricci curvature.
Ann.\ Inst.\ Fourier (Grenoble) {\bf 59} (2009), 563--573.

\bibitem[GS]{GS}
Y.~Ge and Z.~Shen, Eigenvalues and eigenfunctions of metric measure manifolds.
Proc.\ London Math.\ Soc.\ (3) {\bf 82} (2001), 725--746.

\bibitem[Gi1]{G-split}
N.~Gigli, The splitting theorem in non-smooth context.
Preprint (2013). Available at {\sf arXiv:1302.5555}

\bibitem[Gi2]{G-nsdg}
N.~Gigli, Nonsmooth differential geometry
-- An approach tailored for spaces with Ricci curvature bounded from below.
Mem.\ Amer.\ Math.\ Soc.\ (to appear). Available at {\sf arXiv:1407.0809}

\bibitem[Gr]{Gr}
M.~Gromov,
Metric structures for Riemannian and non-Riemannian spaces.
Based on the 1981 French original. With appendices by M.~Katz, P.~Pansu and S.~Semmes.
Translated from the French by Sean Michael Bates.
Birkh\"auser Boston, Inc., Boston, MA, 1999.

\bibitem[GM]{GM}
M.~Gromov and V.~D.~Milman,
Generalization of the spherical isoperimetric inequality to uniformly convex Banach spaces.
Compositio Math.\ {\bf 62} (1987), 263--282.

\bibitem[KLS]{KLS}
R.~Kannan, L.~Lov\'asz and M.~Simonovits,
Isoperimetric problems for convex bodies and a localization lemma.
Discrete Comput.\ Geom.\ {\bf 13} (1995), 541--559.

\bibitem[Kl]{Kl}
B.~Klartag, Needle decompositions in Riemannian geometry.
Mem.\ Amer.\ Math.\ Soc.\ (to appear). Available at {\sf arXiv:1408.6322}

\bibitem[KM]{KM}
A.~V.~Kolesnikov and E.~Milman,
Poincar\'e and Brunn--Minkowski inequalities on weighted Riemannian manifolds with boundary.
Preprint (2013). Available at {\sf arXiv:1310.2526}

\bibitem[Ku]{Ku}
K.~Kuwada, Duality on gradient estimates and Wasserstein controls.
J.\ Funct.\ Anal.\ {\bf 258} (2010), 3758--3774.

\bibitem[Led]{Led}
M.~Ledoux, The concentration of measure phenomenon.
Mathematical Surveys and Monographs, {\bf 89}. American Mathematical Society, Providence, RI, 2001.

\bibitem[Lee]{Lee}
P.~W.~Y.~Lee, Displacement interpolations from a Hamiltonian point of view.
J.\ Funct.\ Anal.\ {\bf 265} (2013), 3163--3203.

\bibitem[L\'e1]{Lev1}
P.~L\'evy, Le\c{c}ons d'analyse fonctionnelle. Gauthier-Villars, Paris, 1922.

\bibitem[L\'e2]{Lev2}
P.~L\'evy, Probl\`emes concrets d'analyse fonctionnelle.
Avec un compl\'ement sur les fonctionnelles analytiques par F.~Pellegrino (French).
2d ed. Gauthier-Villars, Paris, 1951.

\bibitem[Li]{Li}
A.~Lichnerowicz, Vari\'et\'es riemanniennes \`a tenseur C non n\'egatif (French).
C.\ R.\ Acad.\ Sci.\ Paris S\'er.\ A-B {\bf 271} (1970), A650--A653.

\bibitem[LV]{LV}
J.~Lott and C.~Villani,
Ricci curvature for metric-measure spaces via optimal transport.
Ann.\ of Math.\ {\bf 169} (2009), 903--991.

\bibitem[LS]{LS}
L.~Lov\'asz and M.~Simonovits,
Random walks in a convex body and an improved volume algorithm.
Random Structures Algorithms {\bf 4} (1993), 359--412.

\bibitem[Ma]{Ma}
C.~H.~Mai,
On Riemannian manifolds with positive weighted Ricci curvature of negative effective dimension.
In preparation (2017).

\bibitem[Mi1]{Misharp}
E. Milman,
Sharp isoperimetric inequalities and model spaces for curvature-dimension-diameter condition.
J.\ Eur.\ Math.\ Soc.\ (JEMS) {\bf 17} (2015), 1041--1078.

\bibitem[Mi2]{Mineg}
E.~Milman,
Beyond traditional curvature-dimension I:
new model spaces for isoperimetric and concentration inequalities in negative dimension.
Trans.\ Amer.\ Math.\ Soc.\ {\bf 369} (2017), 3605--3637.

\bibitem[Mi3]{Miharm}
E.~Milman, Harmonic measures on the sphere via curvature-dimension.
Ann.\ Fac.\ Sci.\ Toulouse Math.\ (to appear). Available at {\sf arXiv:1505.04335}

\bibitem[MN]{MN}
A.~Mondino and A.~Naber,
Structure theory of metric-measure spaces with lower Ricci curvature bounds~I.
Preprint (2014). Available at {\sf arXiv:1405.2222}

\bibitem[Oh1]{Omcp}
S.~Ohta,
On the measure contraction property of metric measure spaces.
Comment.\ Math.\ Helv.\ {\bf 82} (2007), 805--828.

\bibitem[Oh2]{Ouni}
S.~Ohta, Uniform convexity and smoothness, and their applications in Finsler geometry.
Math.\ Ann.\ {\bf 343} (2009), 669--699.

\bibitem[Oh3]{Oint}
S.~Ohta, Finsler interpolation inequalities.
Calc.\ Var.\ Partial Differential Equations {\bf 36} (2009), 211--249.

\bibitem[Oh4]{Osurv}
S.~Ohta,
Optimal transport and Ricci curvature in Finsler geometry.
Probabilistic approach to geometry, 323--342,
Adv.\ Stud.\ Pure Math., {\bf 57}, Math.\ Soc.\ Japan, Tokyo, 2010.

\bibitem[Oh5]{ORand}
S.~Ohta, Vanishing S-curvature of Randers spaces.
Differential Geom.\ Appl.\ {\bf 29} (2011), 174--178.

\bibitem[Oh6]{Ogren}
S.~Ohta, Ricci curvature, entropy, and optimal transport.
Optimal transportation, 145--199,
London Math.\ Soc.\ Lecture Note Ser., {\bf 413}, Cambridge Univ.\ Press, Cambridge, 2014.

\bibitem[Oh7]{Oham}
S.~Ohta, On the curvature and heat flow on Hamiltonian systems.
Anal.\ Geom.\ Metr.\ Spaces {\bf 2} (2014), 81--114.

\bibitem[Oh8]{Osplit}
S.~Ohta, Splitting theorems for Finsler manifolds of nonnegative Ricci curvature.
J.\ Reine Angew.\ Math.\ {\bf 700} (2015), 155--174.

\bibitem[Oh9]{Oneg}
S.~Ohta, $(K,N)$-convexity and the curvature-dimension condition for negative $N$.
J.\ Geom.\ Anal.\ {\bf 26} (2016), 2067--2096.

\bibitem[Oh10]{Oneedle}
S.~Ohta, Needle decompositions and isoperimetric inequalities in Finsler geometry.
J.\ Math.\ Soc.\ Japan (to appear). Available at {\sf arXiv:1506.05876}

\bibitem[Oh11]{Oisop}
S.~Ohta, A semigroup approach to Finsler geometry: Bakry--Ledoux's isoperimetric inequality.
Preprint (2016). Available at {\sf arXiv:1602.00390}

\bibitem[Oh12]{Ofunc}
S.~Ohta, Some functional inequalities on non-reversible Finsler manifolds.
Proc.\ Indian Acad.\ Sci.\ Math.\ Sci.\ (to appear). Available at {\sf arXiv:1701.05704}

\bibitem[OP]{OP}
S.~Ohta and M.~P\'alfia,
Gradient flows and a Trotter--Kato formula of semi-convex functions on CAT(1)-spaces.
Amer.\ J.\ Math.\ (to appear). Available at {\sf arXiv:1411.1461}

\bibitem[OS1]{OShf}
S.~Ohta and K.-T.~Sturm, Heat flow on Finsler manifolds.
Comm.\ Pure Appl.\ Math.\ {\bf 62} (2009), 1386--1433.

\bibitem[OS2]{OSnc}
S.~Ohta and K.-T.~Sturm, Non-contraction of heat flow on Minkowski spaces.
Arch.\ Ration.\ Mech.\ Anal.\ {\bf 204} (2012), 917--944.

\bibitem[OS3]{OSbw}
S.~Ohta and K.-T.~Sturm, Bochner--Weitzenb\"ock formula and Li--Yau estimates on Finsler manifolds.
Adv.\ Math.\ {\bf 252} (2014), 429--448.

\bibitem[OV]{OV}
F.~Otto and C.~Villani,
Generalization of an inequality by Talagrand and links with the logarithmic Sobolev inequality.
J.\ Funct.\ Anal.\ {\bf 173} (2000), 361--400.

\bibitem[PW]{PW}
L.~E.~Payne and H.~F.~Weinberger,
An optimal Poincar\'e inequality for convex domains.
Arch.\ Rational Mech.\ Anal.\ {\bf 5} (1960), 286--292.

\bibitem[Qi]{Qi}
Z.~Qian, Estimates for weighted volumes and applications.
Quart.\ J.\ Math.\ Oxford Ser.\ (2) {\bf 48} (1997), 235--242.

\bibitem[vRS]{vRS}
M.-K.~von Renesse and K.-T.~Sturm,
Transport inequalities, gradient estimates, entropy and Ricci curvature.
Comm.\ Pure Appl.\ Math.\ {\bf 58} (2005), 923--940.

\bibitem[Sal]{Sal}
L.~Saloff-Coste,
Uniformly elliptic operators on Riemannian manifolds.
J.\ Differential Geom.\ {\bf 36} (1992), 417--450.

\bibitem[SS]{SS}
Y.-B.~Shen and Z.~Shen, Introduction to modern Finsler geometry.
World Scientific Publishing Co., Singapore, 2016.

\bibitem[Sh1]{Shvol}
Z.~Shen,
Volume comparison and its applications in Riemann-Finsler geometry.
Adv.\ Math.\ {\bf 128} (1997), 306--328.

\bibitem[Sh2]{Shlec}
Z.~Shen, Lectures on Finsler geometry.
World Scientific Publishing Co., Singapore, 2001.

\bibitem[St1]{StI}
K.-T.~Sturm, On the geometry of metric measure spaces.~I.
Acta Math.\ {\bf 196} (2006), 65--131.

\bibitem[St2]{StII}
K.-T.~Sturm, On the geometry of metric measure spaces.~II.
Acta Math.\ {\bf 196} (2006), 133--177.

\bibitem[SC]{SC}
V.~N.~Sudakov and B.~S.~Cirel'son,
Extremal properties of half-spaces for spherically invariant measures. (Russian)
Problems in the theory of probability distributions, II.
Zap.\ Nau\v{c}n.\ Sem.\ Leningrad.\ Otdel.\ Mat.\ Inst.\ Steklov.\ (LOMI) {\bf 41} (1974), 14--24, 165.

\bibitem[Vi]{Vi}
C.~Villani, Optimal transport, old and new.
Springer-Verlag, Berlin, 2009.

\bibitem[WX]{WX}
G.~Wang and C.~Xia,
A sharp lower bound for the first eigenvalue on Finsler manifolds.
Ann.\ Inst.\ H.\ Poincar\'e Anal.\ Non Lin\'eaire {\bf 30} (2013), 983--996.

\bibitem[WW]{WW}
G.~Wei and W.~Wylie, Comparison geometry for the Bakry--Emery Ricci tensor.
J.\ Differential Geom.\ {\bf 83} (2009), 377--405.

\bibitem[Wy]{Wy}
W.~Wylie, A warped product version of the Cheeger--Gromoll splitting theorem.
Trans.\ Amer.\ Math.\ Soc.\ (to appear).  Available at {\sf arXiv:1506.03800}

\bibitem[Xi]{Xi}
C.~Xia, Local gradient estimate for harmonic functions on Finsler manifolds.
Calc.\ Var.\ Partial Differential Equations {\bf 51} (2014), 849--865.

\bibitem[YH]{YH}
S.-T.~Yin and Q.~He, The first eigenvalue of Finsler $p$-Laplacian.
Differential Geom.\ Appl.\ {\bf 35} (2014), 30--49.

\end{thebibliography}
\end{document}